\documentclass[11pt]{article}

\usepackage{amsmath,amsfonts,amsthm,amssymb,graphicx,tikz,tikz-cd,url,hyperref}
\usepackage[all,2cell,ps]{xy}

\theoremstyle{plain}
\newtheorem{thm}{Theorem}[section]

\newtheorem{lem}[thm]{Lemma}
\newtheorem{prop}[thm]{Proposition}

\newtheorem{cor}[thm]{Corollary}

\newtheorem{proj}{Project}
\newtheorem{qtn}[proj]{Question}

\theoremstyle{definition}

\newtheorem{rem}[thm]{Remark}

\theoremstyle{remark}

\newcommand{\bbB}{\mathbb{B}}
\newcommand{\bbC}{\mathbb{C}}

\newcommand{\bbL}{\mathbb{L}}

\newcommand{\bbN}{\mathbb{N}}

\newcommand{\bbP}{\mathbb{P}}
\newcommand{\bbQ}{\mathbb{Q}}
\newcommand{\bbR}{\mathbb{R}}

\newcommand{\bbZ}{\mathbb{Z}}

\newcommand{\bfG}{\mathbf{G}}

\newcommand{\calL}{\mathcal{L}}

\newcommand{\calN}{\mathcal{N}}
\newcommand{\calO}{\mathcal{O}}

\newcommand{\calR}{\mathcal{R}}

\newcommand{\fraka}{\mathfrak{a}}

\newcommand{\frakg}{\mathfrak{g}}

\newcommand{\frakk}{\mathfrak{k}}

\newcommand{\frakn}{\mathfrak{n}}

\newcommand{\al}{\alpha}
\newcommand{\gam}{\gamma}
\newcommand{\Gam}{\Gamma}

\newcommand{\de}{\delta}
\newcommand{\Del}{\Delta}
\newcommand{\ep}{\epsilon}

\newcommand{\lam}{\lambda}
\newcommand{\Lam}{\Lambda}
\newcommand{\sig}{\sigma}
\newcommand{\Sig}{\Sigma}

\newcommand{\om}{\omega}

\DeclareMathOperator{\U}{U}

\DeclareMathOperator{\M}{M}
\DeclareMathOperator{\SL}{SL}
\DeclareMathOperator{\PSL}{PSL}
\DeclareMathOperator{\GL}{GL}
\DeclareMathOperator{\PGL}{PGL}

\DeclareMathOperator{\PU}{PU}
\DeclareMathOperator{\SU}{SU}

\DeclareMathOperator{\Id}{Id}

\DeclareMathOperator{\ad}{ad}
\DeclareMathOperator{\Ad}{Ad}

\DeclareMathOperator{\Hom}{Hom}
\DeclareMathOperator{\End}{End}

\newcommand{\bs}{\backslash}

\newcommand{\lra}{\longrightarrow}

\newcommand{\ssm}{\smallsetminus}

\newcommand{\conj}{\overline}
\newcommand{\wh}{\widehat}
\newcommand{\wt}{\widetilde}

\newenvironment{pf}{\begin{proof}}{\end{proof}}

\newenvironment{enum}{\begin{enumerate}}{\end{enumerate}}

\usepackage{tikz-cd}
\allowdisplaybreaks

\tikzcdset{scale cd/.style={every label/.append style={scale=#1},
    cells={nodes={scale=#1}}}}

\makeatletter
\let\@@pmod\pmod
\DeclareRobustCommand{\pmod}{\@ifstar\@pmods\@@pmod}
\def\@pmods#1{\mkern4mu({\operator@font mod}\mkern 6mu#1)}
\makeatother

\usetikzlibrary{arrows}

\begin{document}

\title{Residually finite lattices in $\wt{\PU(2,1)}$ and fundamental groups of smooth projective surfaces}

\author{Matthew Stover\\ \small{Temple University}\\ \small{\textsf{mstover@temple.edu}} \and Domingo Toledo\\ \small{University of Utah}\\ \small{\textsf{toledo@math.utah.edu}}}

\date{\today}

\maketitle

\begin{center}
\vspace{-1.5em}\emph{To Gopal Prasad in celebration of his 75th birthday}
\end{center}

\begin{abstract}
This paper studies residual finiteness of lattices in the universal cover of $\PU(2,1)$ and applications to the existence of smooth projective varieties with fundamental group a cocompact lattice in $\PU(2,1)$ or a finite covering of it. First, we prove that certain lattices in the universal cover of $\PU(2,1)$ are residually finite. To our knowledge, these are the first such examples. We then use residually finite central extensions of torsion-free lattices in $\PU(2,1)$ to construct smooth projective surfaces that are not birationally equivalent to a smooth compact ball quotient but whose fundamental group is a torsion-free cocompact lattice in $\PU(2,1)$.
\end{abstract}

\section{Introduction}\label{sec:Intro}

The purpose of this paper is to study residual finiteness of lattices in the universal cover of $\PU(2,1)$ and applications to the existence of smooth projective varieties with fundamental group a cocompact lattice in $\PU(2,1)$ or a finite covering of it. This follows a general theme used by the second author \cite{ToledoNotRF}, and in a different way by Nori (unpublished) and Catanese--Koll\'ar \cite{Trento}, to build smooth projective varieties with fundamental group that is not residually finite. First, we prove that certain lattices in the universal cover of $\PU(2,1)$ are residually finite. To our knowledge, these are the first such examples. We then use residually finite central extensions of torsion-free lattices in $\PU(2,1)$ to construct smooth projective surfaces that are not birationally equivalent to a smooth compact ball quotient but whose fundamental group is a torsion-free cocompact lattice in $\PU(2,1)$.

\medskip

Let $G$ be an adjoint simple Lie group of hermitian type. Then $\bbZ \le \pi_1(G)$, and the universal cover $\wt{G}$ of $G$ fits in a central extension
\[
1 \lra \bbZ \lra \wt{G} \lra \conj{G} \lra 1
\]
for a certain finite cover $\conj{G}$ of $G$. Adjoint simple groups are linear, and it follows from a well-known theorem of Malcev \cite{MalcevLinear} that any lattice $\Gam < G$ is residually finite. However, $\wt{G}$ is not linear and so it is not clear whether or not the preimage $\wt{\Gam}$ of $\Gam$ in $\wt{G}$ is residually finite. When $\Gam$ is an arithmetic subgroup of an absolutely simple, simply connected algebraic group $\bfG$ defined over a number field $K$ and $\Gam$ has the congruence subgroup property, Deligne \cite{DeligneRF} showed that the preimage $\wt{\Gam}$ of $\Gam$ in $\wt{G} = \wt{\bfG}(\bbR)$ need not be residually finite. If $z$ is a generator for $\bbZ \le \pi_1(\bfG(\bbR)) < \wt{G}$ and $n$ is the order of the \emph{absolute metaplectic kernel} $\M(\varnothing, \bfG)$, Deligne proved that $z^n$ maps trivially to every finite quotient of $\wt{\Gam}$. See work of Prasad and Rapinchuk \cite[p.\ 92]{Prasad-Rapinchuk} for a precise description of the connection between residual finiteness of $\wt{\Gam}$ and $|\M(\varnothing, \bfG)|$ and for a complete description of $\M(\varnothing, \bfG)$. One then sees that the preimage of $\Gam$ in any connected $d$-fold cover of $\bfG(\bbR)$ with $d > n$ also cannot be residually finite.

There are many known examples of lattices in $\PU(n,1)$ without the congruence subgroup property, and it is widely expected to fail for all such lattices. Consequently, the following question seems interesting in contrast with the situation in higher-rank.

\begin{qtn}\label{qtn:RF?}
Let $G = \PU(n,1)$, $\wt{G}$ be its universal cover, $\Gam < G$ be a lattice, and $\wt{\Gam}$ be the preimage of $\Gam$ in $\wt{G}$. Is $\wt{\Gam}$ residually finite? Is $\wt{\Gam}$ linear?
\end{qtn}

We briefly recall that $\wt{\Gam}$ also has a natural geometric interpretation when $\Gam$ is torsion-free and cocompact. Let $\bbB^n$ denote complex hyperbolic $n$-space. Then $\wt{\Gam}$ is the fundamental group of the complement of the zero section inside the total space of the canonical bundle of $\Gam \bs \bbB^n$. It can equivalently be considered as the fundamental group of the space
\[
\wh{\Gam} \bs \SU(n,1) / \SU(n),
\]
where $\wh{\Gam}$ is the preimage of $\Gam$ in $\SU(n,1)$ \cite[\S 8.9]{Kollar}. See \S\ref{ssec:HomogSp} for more on the connection with $\SU(n,1) / \SU(n)$.

We also note another motivation for Question~\ref{qtn:RF?} from geometric group theory. If there is a cocompact lattice $\wt{\Gam}$ in $\wt{G}$ that is not residually finite, then, as in Deligne's work \cite{DeligneRF}, one could find a lattice in a finite cover of $\PU(n,1)$ that is not residually finite. This would be a finite central extension of a cocompact lattice in $\PU(n,1)$, hence it would be a Gromov hyperbolic group that is not residually finite. Whether or not Gromov hyperbolic groups are residually finite is a long-standing open problem.

It is known in the case $n=1$ that Question~\ref{qtn:RF?} has a positive answer. In other words, lattices in the universal cover of $\PU(1,1) \cong \PSL_2(\bbR)$ are both residually finite and linear. Very briefly, the most well-known proof goes as follows. First, passing to a subgroup of finite index, we can assume that $\wt{\Gam}$ is the fundamental group of the unit tangent bundle of a hyperbolic surface. Either geometrically or using a presentation for $\wt{\Gam}$, one can produce a homomorphism from $\wt{\Gam}$ to a two-step nilpotent group (in fact, an integer Heisenberg group) for which the restriction to the center of $\wt{\Gam}$ is injective. Since $\Gam$ is residually finite, this suffices to conclude that $\wt{\Gam}$ is also residually finite. See \cite[\S IV.48]{delaHarpe} for precise details from both the geometric and algebraic perspectives, and \S\ref{sec:RF} for more on why finding such a quotient suffices to prove residual finiteness.

The first aim of this paper is to construct the first examples of lattices in $\PU(2,1)$ whose lift to the universal cover is residually finite. We in fact develop two methods for proving residual finiteness, each generalizing a variant of the classical argument for the case $n=1$. Both strategies eventually intersect in the endgame of finding a two-step nilpotent quotient of $\wt{\Gam}$ for which the central element survives to have infinite order. For torsion-free cocompact lattices, the two methods end up being effectively equivalent, but one or the other can be used more easily in different settings. We will prove:

\begin{thm}\label{thm:ExsExist}
Let $G = \PU(2,1)$. There are lattices $\Gam < G$ so that the preimage of $\Gam$ in any connected cover of $G$ is residually finite. One can take $\Gam$ to be uniform or nonuniform.
\end{thm}

For a uniform arithmetic lattice, one can take the fundamental group of the Cartwright--Steger surface \cite{CartwrightSteger}. We prove this indirectly, using a $21$-fold cover studied by the first author \cite{StoverHurwitz} and further explored by Dzambic and Roulleau \cite{DzambicRoulleau}. See Theorem~\ref{thm:StoverRF}. A nonuniform arithmetic example is the Picard modular group over the Eisenstein integers. All of our examples are commensurable with Deligne--Mostow lattices \cite{DeligneMostow1, MostowSigma}. See \S\ref{sec:Ex}.

The first method of constructing examples, presented in \S\ref{sec:Lift}, is more computational in nature. Replacing $\Gam$ with its preimage in $\SU(n,1)$ and using the fact that $\Gam \bs \SU(n,1) / \SU(n)$ has fundamental group $\wt{\Gam}$, we describe a method for constructing a finite presentation for $\wt{\Gam}$ using one for $\Gam$. This is used in particular to provide the nonuniform examples in Theorem~\ref{thm:ExsExist}. Our second method, inspired by ideas going back at least to work of Sullivan in the 70s (see \cite[Ch.\ 3]{ABCKT}), uses the cup product structure on $H^*(\Gam \bs \bbB^n, \bbQ)$. The following, a special case of Corollary~\ref{cor:RFCentralKahler} and a consequence of a more general theorem about principal $\U(1)$ bundles over aspherical manifolds proved in \S\ref{sec:Beauville}, gives a sharp answer as to when the classical strategy for $n = 1$ can be used for $n > 1$ to prove residual finiteness of $\wt{\Gam}$.

\begin{thm}\label{thm:GeneralCup}
Let $\Gam < \PU(n,1)$ be a torsion-free cocompact lattice and $X = \Gam \bs \bbB^n$. If $\wt{\Gam}$ is the preimage of $\Gam$ in the universal cover of $\PU(n,1)$, then there is a two-step nilpotent quotient of $\wt{\Gam}$ that is injective on the center of $\wt{\Gam}$ if and only if the K\"ahler class on $X$ is in the image of the cup product map
\[
c_\bbQ : \bigwedge\nolimits^2 H^1(X, \bbQ) \lra H^2(X, \bbQ).
\]
\end{thm}

If $\Gam < \PU(n,1)$ is an arithmetic lattice defined by a hermitian form over a CM field, it is known that there is a torsion-free subgroup of finite index $\Gam^\prime \le \Gam$ so that $X^\prime = \Gam^\prime \bs \bbB^n$ has $H^1(X^\prime, \bbQ)$ nontrivial \cite{Kazhdan, Shimura}. In fact, Clozel proved that one can pass to a possible smaller subgroup and assume that the image of $c_\bbQ : \bigwedge\nolimits^2 H^1(X^\prime, \bbQ) \to H^2(X^\prime, \bbQ)$ is nontrivial \cite{Clozel}. It seems more subtle to ensure that the K\"ahler class is eventually in the image. A positive answer to the following question would, via Theorem~\ref{thm:GeneralCup}, prove residual finiteness of preimages of these arithmetic lattices in the universal cover of $\PU(n,1)$.

\begin{qtn}\label{qtn:FindCup}
Let $\Gam < \PU(n,1)$ be a cocompact arithmetic lattice defined by a hermitian form over a CM field. Is there a torsion-free finite index subgroup $\Gam^\prime \le \Gam$ so that $X^\prime = \Gam^\prime \bs \bbB^n$ has the property that the K\"ahler class on $X^\prime$ is in the image of the map $c_\bbQ : \bigwedge\nolimits^2 H^1(X^\prime, \bbQ) \to H^2(X^\prime, \bbQ)$?
\end{qtn}

\begin{rem}
After submitting this paper, we answered Question~\ref{qtn:FindCup} in the positive and hence gave a positive answer to Question~\ref{qtn:RF?} for cocompact arithmetic lattices in $\PU(n,1)$ of simple type \cite{StoverToledo2}. This was independently proved by Richard Hill using similar but ultimately different methods \cite{Hill}.
\end{rem}

In the second half of this paper we show how residual finiteness can be used to construct algebraic surfaces that are not ball quotients but have fundamental group isomorphic to a lattice in $\PU(2,1)$.  To state our results, in the rest of this section we use the following notation: $G =\PU(2,1)$, $d>1$ is an integer, $G_d$ is the connected $d$-fold cover of $G$, $\Gamma < G$ is a torsion-free, cocompact lattice, and $\Gamma_d$ is the preimage of $\Gamma$ in $G_d$.  So $\Gamma_d$ fits into a central extension:
\begin{equation}\label{eq:gamd}
1 \lra \bbZ / d \lra \Gam_d \overset{\pi_d}{\lra} \Gam \lra 1,
\end{equation}
and $\Gam_d$ is a cocompact lattice in $G_d$.
 
The construction of Nori and Catanese--Koll\'ar mentioned above can be used to produce smooth projective varieties with fundamental group isomorphic to $\Gamma_d$. In \S\ref{sec:Realize} we will explain in detail one version of this construction, as sketched in \cite[Ex.~8.15]{ABCKT}. Given $\Gam$ and $d$, this construction produces  a smooth projective surface $M(\Gamma,d)$ with $\pi_1(M(\Gamma,d))\cong \Gamma_d$. While the interest in \cite[Ex.~8.15]{ABCKT} was to apply the construction to non-residually finite lattices, we are interested in applying it in the opposite situation.  Our main observation is that, when $\Gam_d$ is residually finite, $M(\Gam,d)$ has \'etale covers with fundamental group a subgroup of finite index in $\Gam$.  These covers have fundamental group isomorphic to a lattice in $G$, but are not ball quotients. We summarize the situation in the following theorem:

\begin{thm}\label{thm:dFold}
In the notation established above:
\begin{enumerate}

\item For each $\Gam$ and $d$ there is a smooth projective surface $M(\Gam, d)$ with $\pi_1(M(\Gam,d)) \cong \Gam_d$ and $\pi_2(M(\Gam,d)) \ne 0$.

\item If $\Gam_d$ is residually finite, there exist subgroups $\Lam\subset \Gam$ of finite index over which \eqref{eq:gamd} has a section $s:\Lam\to\Gam_d$.  

\item If $M' = M'(\Lam, s.d)$ denotes the cover of $M = M(\Gam,d)$ corresponding to the subgroup $s(\Lam) < \Gam_d$, then $M'$ is an \'etale cover of $M(\Gam,d)$ with $\pi_1(M')\cong\Lam$,  $\pi_2(M')\ne 0$, and $M'$ is not birationally equivalent to a ball quotient.

\end{enumerate}
\end{thm}

When $\Gam_d$ is residually finite, the third part of this theorem provides examples of smooth projective surfaces with fundamental group isomorphic to a lattice in $G$ but not birationally equivalent  to a ball quotient.  It is known that such surfaces  exist, but it is useful  to have constructions.  In a more general setting, it is known that for any smooth projective variety of dimension at least two there are many varieties with the same fundamental group. The earliest construction seems to be by Cornalba \cite{Cornalba}, discussed on \cite[p.~36]{Fulton}, and another construction is given in \cite[\S 1]{catanese}. Common ingredients of these constructions, including the one given here, is that they appear as branched covers, and their properties are verified using Lefschetz theorems.

Recent work of Troncoso and Urz\'ua \cite{TroncosoUrzua} constructs many more smooth projective surfaces with the same fundamental group as a ball quotient. A key distinction between our work and \cite{TroncosoUrzua} is that we build surfaces whose fundamental group is a lattice in a finite cover of $\PU(2,1)$, which their work does not do, but \cite{TroncosoUrzua} construct surfaces with the same fundamental group as a given ball quotient with Chern slopes dense in $[1,3]$, which is much more robust than what our methods produce (e.g., see Remark~\ref{rem:FakeChern}). In contrast, a torsion-free nonuniform lattice in $\PU(2,1)$ cannot be the fundamental group of a smooth projective surface \cite[Thm.\ 2]{ToledoExs}. As one specific application of our methods, we will prove the following at the conclusion of \S\ref{sec:Realize}.

\begin{thm}\label{thm:FakeFakes}
Let $X$ be a fake projective plane whose canonical divisor $K_X$ is divisible by three. Then there exists a minimal smooth projective surface $Y$ of general type with $\pi_1(Y) \cong \pi_1(X)$ that is not a fake projective plane.
\end{thm}

Recall that fake projective planes were classified in combined work of Prasad--Yeung \cite{PrasadYeung, PrasadYeungErrat} and Cartwright--Steger \cite{CartwrightSteger}. Of the 100 fake projective planes, all but 8 have the property that $K_X$ is divisible by three \cite[\S A.1]{PrasadYeungErrat}. However, we were not able to answer the following special case of Question~\ref{qtn:RF?}.

\begin{qtn}
Let $X$ be a fake projective plane and $\wt{\Gam}$ be the preimage of $\pi_1(X)$ in the universal cover of $\PU(2,1)$. Is $\wt{\Gam}$ residually finite?
\end{qtn}

\medskip

This paper is organized as follows. In \S\ref{sec:RF} we begin with some basic facts about residual finiteness. Then \S\ref{sec:Struc} describes explicit coordinates on the universal cover of $\SU(n,1)$ that are used in \S\ref{sec:Lift} to describe our method for lifting a presentation of a lattice in $\SU(n,1)$ to a presentation of its lift to the universal cover. The proof of Theorem~\ref{thm:GeneralCup} and generalizations are contained in \S\ref{sec:Beauville}. In \S\ref{sec:Samepi1}, we prove that a smooth projective surface with $\pi_2$ trivial and fundamental group a lattice in $\PU(2,1)$ is biholomorphic to a ball quotient. We also sketch a proof of a related theorem due to Thomas Delzant, whose argument was communicated to us by Pierre Py: If $X$ is a compact K\"ahler surface with $\pi_1(X)$ not commensurable with a surface group, then $X$ is aspherical if and only if $\pi_2(X)$ is trivial. Then \S\ref{sec:Realize} proves Theorem~\ref{thm:dFold} and related results. Finally, in \S\ref{sec:Ex} we discuss examples that suffice to prove Theorem~\ref{thm:ExsExist}.

\subsubsection*{Acknowledgments} The authors thank Gopal Prasad for helpful correspondence regarding the metaplectic kernel, and Fabrizio Catanese, Pierre Py and Burt Totaro for helpful comments. We also thank Thomas Delzant for allowing us to include a sketch of his proof of Theorem~\ref{thm:delzant}. Finally, we thank the referee for helpful suggestions. Stover was partially supported by Grant Number DMS-1906088 from the National Science Foundation.

\section{Residual finiteness}\label{sec:RF}

Here, we collect some basic facts about residual finiteness. Recall that a group $\Gam$ is residually finite if, given any nontrivial $\gam \in \Gam$, there exists a finite group $F$ and a homomorphism $\rho : \Gam \to F$ so that $\rho(\gam)$ is not the identity element of $F$. Finitely generated linear groups are residually finite by a famous theorem of Malcev \cite{MalcevLinear} (also see \cite[\S 2.1]{LongReid}). We begin with a basic and well-known fact about residual finiteness; see for instance \cite[Lem.\ 2.1.3]{LongReid} for a proof.

\begin{lem}\label{lem:CommensurableRF}
Suppose that $\Gam$ is a finitely generated group and $\Del \le \Gam$ is a finite index subgroup. Then $\Gam$ is residually finite if and only if $\Del$ is residually finite.
\end{lem}

We will also need the following simple lemma.

\begin{lem}\label{lem:CentralRF}
For $n \ge 2$, suppose that
\[
1 \lra \bbZ / n \lra \Gam_n \overset{\pi_n}{\lra} \Gam \lra 1
\]
is a central exact sequence. If $\Gam_n$ is residually finite, then there exists a finite index subgroup $\Gam^\prime \le \Gam$ that admits a section to $\Gam_n$ under $\pi_n$. Consequently, $\Gam_n$ has a finite index subgroup that is isomorphic to a finite index subgroup of $\Gam$.
\end{lem}

\begin{pf}
Suppose that $\Gam_n$ is residually finite and that $\sig$ is a generator for $\ker(\pi_n)$. For each $1 \le j \le n-1$ there exists a finite group $F_j$ and a homomorphism $\rho_j : \Gam_n \to F_j$ so that $\rho(\sig^j)$ is nontrivial. The product homomorphism
\[
\rho = \left(\rho_1 \times \cdots \times \rho_{n-1}\right) : \Gam_n \lra \prod_{j = 1}^{n-1} F_j
\]
then has the property that $\rho(\sig^j)$ is nontrivial for each $1 \le j \le n-1$. If $\Gam^\prime = \ker(\rho)$, then $\Gam^\prime \cap \langle \sig \rangle = \{1\}$. Since $\langle \sig \rangle = \ker(\pi_n)$, we have that $\pi_n|_{\Gam^\prime}$ is an isomorphism onto its image. This proves the lemma.
\end{pf}

\begin{cor}\label{cor:PUlift}
For any $n \ge 1$, let $\Gam < \SU(n,1)$ be a finitely generated subgroup. Then there is a finite index subgroup $\Gam^\prime \le \Gam$ so that projection $\SU(n,1) \to \PU(n,1)$ restricts to an isomorphism on $\Gam^\prime$. In particular, every lattice in $\SU(n,1)$ contains a finite index subgroup isomorphic to a lattice in $\PU(n,1)$.
\end{cor}

\begin{pf}
Since $\SU(n,1)$ is linear, $\Gam$ is residually finite. Let $\sig$ be a generator for the intersection of $\Gam$ with the center of $\SU(n,1)$. The center of $\SU(n,1)$ is isomorphic to $\bbZ / (n + 1)$, so $\langle \sig \rangle$ is cyclic. Therefore if $\conj{\Gam}$ denotes the image of $\Gam$ in $\PU(n,1)$, we have a central exact sequence
\[
1 \lra \langle \sig \rangle \lra \Gam \lra \conj{\Gam} \lra 1
\]
to which Lemma~\ref{lem:CentralRF} applies, giving the desired $\Gam^\prime$.
\end{pf}

\begin{lem}\label{lem:RFJustCenter}
Suppose that $\Gam$ is a residually finite group and that
\[
1 \lra Z \lra \wt{\Gam} \lra \Gam \lra 1
\]
is an extension. If there is a homomorphism $\rho : \wt{\Gam} \to H$ onto a residually finite group $H$ so that $\rho|_Z$ is injective, then $\wt{\Gam}$ is residually finite.
\end{lem}

\begin{pf}
Given a nontrivial element $\al \in \wt{\Gam}$, if the projection $\conj{\al}$ of $\al$ to $\Gam$ is nontrivial, then there is a map from $\Gam$ to a finite group where the image of $\conj{\al}$ is nontrivial, hence there is a homomorphism from $\wt{\Gam}$ to a finite group so that the image of $\al$ is nontrivial. If $\conj{\al}$ is trivial, then $\al \in Z$. Then $\rho(\al) \in H$ is nontrivial, and there is a finite quotient of $H$, which we consider as a finite quotient of $\wt{\Gam}$, where the image of $\al$ is nontrivial. This proves that $\wt{\Gam}$ is residually finite.
\end{pf}

Finitely generated nilpotent groups are known to be residually finite. By Lemma~\ref{lem:CommensurableRF}, this follows from the fact that finitely generated nilpotent groups contain a torsion-free subgroup of finite index \cite[Thm.~3.21]{Hirsch} combined with Malcev's famous results that finitely generated torsion-free nilpotent groups are linear \cite{MalcevNilpotent} and that linear groups are residually finite \cite{MalcevLinear}. The following is then a special case of Lemma~\ref{lem:RFJustCenter}.

\begin{cor}\label{cor:RFNilQuo}
Suppose that $\Gam$, $\wt{\Gam}$, and $Z$ are as in Lemma~\ref{lem:RFJustCenter}. If $\wt{\Gam}$ is finitely generated and there is a nilpotent quotient $\calN$ of $\wt{\Gam}$ so that $Z$ injects into $\calN$, then $\wt{\Gam}$ is residually finite.
\end{cor}

We return to the case of interest for this paper, namely where $\wt{\Gam}$ is a central extension of a finitely generated residually finite group $\Gam$ by $\bbZ$. Corollary~\ref{cor:RFNilQuo} implies that to prove $\wt{\Gam}$ is residually finite, it suffices to find a nilpotent quotient of $\wt{\Gam}$ into which $\bbZ$ injects. We also want to study the central exact sequence obtained by reducing the central element of $\wt{\Gam}$ modulo an integer $d$. It is not immediately clear that residual finiteness of $\wt{\Gam}$ implies residual finiteness of each $\bbZ / d$ central extension, only infinitely many of the extensions. Using separability properties of finitely generated nilpotent groups, we can prove the stronger statement:

\begin{lem}\label{lem:RFmodn}
Suppose that $\Gam$ is a finitely generated residually finite group and that
\[
1 \lra \bbZ \lra \wt{\Gam} \lra \Gam \lra 1
\]
is a central exact sequence associated with $\phi \in H^2(\Gam, \bbZ)$. For $d \in \bbN$, let $\phi_d \in H^2(\Gam, \bbZ/d)$ be the reduction of $\phi$ modulo $d$ and
\[
1 \lra \bbZ / d \lra \Gam_d \lra \Gam \lra 1
\]
be the associated central exact sequence given by reducing the kernel $\langle \sig \rangle$ of $\wt{\Gam} \to \Gam$ modulo $d$. Suppose that there is a nilpotent quotient $\calN$ of $\wt{\Gam}$ that is injective on $\langle \sig \rangle$. Then $\Gam_d$ is residually finite for all $d \in \bbN$.
\end{lem}

\begin{pf}
Let $\sig$ be a generator for the kernel of $\wt{\Gam} \to \Gam$ and $\rho : \Gam \to \calN$ be a nilpotent quotient such that $S = \langle \rho(\sig) \rangle \cong \bbZ$. Since $\sig$ is central in $\wt{\Gam}$, $\rho(\sig)$ is central in $\calN$. In particular, $S_d = \langle \rho(\sig^d) \rangle$ is a normal subgroup of $\calN$ for every natural number $d$.

Malcev proved that finitely generated nilpotent groups are \emph{subgroup separable} \cite{MalcevLERF}, i.e., for every subgroup $H \le \calN$, there is a collection $\{\calN_i\}$ of finite index subgroups of $\calN$ so that $H = \bigcap \calN_i$. See \cite[\S 2.3]{LongReid} for more basic facts about subgroup separability. Note that generally a group is subgroup separable if this condition holds for all \emph{finitely generated} subgroups $H$, but finitely generated nilpotent groups are Noetherian, hence all subgroups are finitely generated.

By \cite[Lem.\ 4.6]{ReidProfinite}, for each $d \in \bbN$ we can find a finite index subgroup $\calN_d$ of $\calN$ so that $\calN_d \cap S = S_d$. Since $S_d$ is normal in $\calN$, we can replace $\calN_d$ by the intersection of all its conjugates in $\calN$ and assume that $\calN_d$ is normal in $\calN$. Then $F_d = \calN / \calN_d$ is a finite quotient of $\wt{\Gam}$ such that the image of $\langle \sig \rangle$ in $F_d$ is isomorphic to $S / (S \cap \calN_d) \cong \bbZ / d$. This means that $\wt{\Gam} \to F_d$ factors through the projection $\wt{\Gam} \to \Gam_d$, and moreover that $\Gam_d \to F_d$ is injective on the central $\bbZ / d$ quotient of $\langle \sig \rangle$. Corollary~\ref{cor:RFNilQuo} then implies that $\Gam_d$ is residually finite, which completes the proof of the lemma.
\end{pf}

\section{Structure of $\SU(n,1)$ and its universal cover}\label{sec:Struc}

This section gives precise coordinates, including an Iwasawa decomposition, for $\SU(n,1)$ and its universal cover. These coordinates will be used to describe a concrete procedure for lifting finitely presented subgroups of $\SU(n,1)$ to $\wt{\SU(n,1)}$. Lastly, we study a certain homogeneous space for each group that will be fundamental for this lifting procedure. We start with some preliminaries on the universal cover $\wt{\U(n)}$ of the unitary group $\U(n)$.

\subsection{Coordinates on $\wt{\U(n)}$}\label{ssec:U(n)hatCoords}

The universal cover $\wt{\U(n)}$ of $\U(n)$ is identified with $\SU(n) \times \bbR$ under the universal covering map
\[
(g, t) \mapsto e^{\pi i t} g.
\]
Composing this with the isomorphism $g \mapsto (g, \det(g)^{-1})$ between $\U(n)$ and $K = \mathrm{S}(\U(n) \times \U(1))$, we see that
\begin{align}
\SU(n) \times \bbR &\lra K \nonumber \\
(g, t) &\longmapsto \left(e^{\pi i t} g, e^{-n \pi i t} \right) \label{eq:U(n)cover}
\end{align}
is the universal cover in coordinates on $K$. Note that
\[
\left\{(e^{-\pi i t} I_n, t)~:~t \in \frac{2}{n} \bbZ\right\}
\]
is the preimage of $(I_n, 1) \in K$ in $\wt{\U(n)}$.

\subsection{Coordinates on $\SU(n,1)$}\label{ssec:SUcoords}

Consider $G = \SU(n,1)$ as the group of special unitary automorphisms of $\bbC^{n+1}$ with respect to the hermitian form with matrix
\[
h = \begin{pmatrix}
0 & 0 & 1 \\
0 & I_{n-1} & 0 \\
1 & 0 & 0
\end{pmatrix}.
\]
Fix the $h$-negative vector
\[
z_0 = \begin{pmatrix} -1 \\ 0 \\ \vdots \\ 0 \\ 1 \end{pmatrix}
\]
of $h$-norm $-2$ and let $K$ be the stabilizer in $G$ of the line spanned by $z_0$, i.e., the elements of $G$ with $z_0$ as an eigenvector. There is an isomorphism $K \cong \mathrm{S}(\U(n) \times \U(1))$ determining coordinates $(g, \xi)$ on $K$ with $g \in \U(n)$ and $\xi = \det(g)^{-1}$. In what follows, we identify $\SU(n) < K$ with $g_0 \mapsto (g_0, 1)$ and $\U(1) < K$ with $\xi \mapsto (\xi I_n, \xi^{-n})$. Then $(g, \xi) z_0 = \xi z_0$ for all $(g, \xi) \in K$, hence $\SU(n)$ is the stabilizer of $z_0$ in $K$ and $K z_0 = \U(1) z_0$.

The $h$-isotropic vector
\[
v_0 = \begin{pmatrix} 1 \\ 0 \\ \vdots \\ 0 \end{pmatrix}
\]
has stabilizer containing the subgroup $B_0 = A_+ N < G$ of upper-triangular matrices with positive real eigenvalues. Specifically,
\[
B_0 = \left\{ \begin{pmatrix} \lam & - \lam {}^t \conj{z} & -\lam |z|^2/2 + i t \\ 0 & I_{n-1} & z \\ 0 & 0 & \lam^{-1} \end{pmatrix}~:~\lam \in \bbR_+, z \in \bbC^{n-1}, t \in \bbR \right\}
\]
where $A_+$ is associated with the diagonal subgroup and $N$ the unipotent radical of $B_0$, which is isomorphic to the $(n-1)$-dimensional Heisenberg group. We now prove that $G = B_0 K$, which allows us to fix an identification of $B_0 = G/K$ with complex hyperbolic space $\bbB^n$. In fact, we show that $G = B_0 K$ gives an Iwasawa decomposition of $G$.

\begin{lem}\label{lem:Iwasawa}
With notation as above, $G = B_0 K$ gives an Iwasawa decomposition of $G$ with respect to the Cartan involution of $G$ with matrix
\[
\Theta = \begin{pmatrix} 0 & 0 & 1 \\ 0 & I_{n-1} & 0 \\ 1 & 0 & 0 \end{pmatrix}.
\]
\end{lem}

\begin{pf}
First, note that $\Theta = h$, hence
\[
{}^t \conj{\Theta} h \Theta = h^3 = h,
\]
and $\det(\Theta) = -1$, so $\Theta \in \U(2,1)$. If $\frakg$ denotes the Lie algebra of $G$, we claim that $\ad(\Theta) \in \End(\frakg)$ is a Cartan involution with the Lie algebra $\fraka$ of $A_+$ as a Cartan subalgebra. Since $\Theta^2 = I_{n+1}$, $\Theta$ normalizes $A_+$, and $\fraka$ is a maximal $\Ad$-semisimple abelian subalgebra of $\frakg$, to show that $\Theta$ is a Cartan involution fixing the Cartan subalgebra $\fraka$, it suffices to show that the $+1$ eigenspace for $\ad(\Theta)$ is precisely the Lie algebra $\frakk$ of $K$.

First, note that $\Theta$ is an element of the natural $\U(2) \times \U(1) < \U(2,1)$ containing $K$. Indeed, $\Theta z_0 = -z_0$. Moreover, $\Theta$ acts trivially on the $h$-orthogonal complement to $z_0$, which is spanned by $e_1 + e_{n+1}$ and $e_2, \dots, e_n$, where $e_i$ is the $i^{th}$ standard basis vector. With respect to coordinates on $\U(2) \times \U(1)$ analogous to our coordinates on $K$, this means that
\[
\Theta = (I_n, -1).
\]
It is clear from these coordinates that conjugation by $\Theta$ acts trivially on $K$, hence $\frakk$ is contained in the $+1$ eigenspace of $\ad(\Theta)$. Conversely, if $\Theta$ commutes with $g \in G$, then $g$ preserves the $-1$ eigenspace of $\Theta$ acting on $\bbC^{n+1}$, which is the line spanned by $z_0$. Thus $z_0$ is an eigenvector of $g$, and so $g \in K$. This proves that $\ad(\Theta)$ is a Cartan involution of $\frakg$ with $+1$ eigenspace $\frakk$.

Taking the root space decomposition of $\frakg$ with respect to $\fraka$, we can choose positive roots so that the associated connected unipotent subgroup of $G$ is $N$. Following \cite[\S VI.3]{Helgason}, we see that $B_0 K = A_+ N K$ is an Iwasawa decomposition of $G$. This completes the proof.
\end{pf}

\subsection{Coordinates on $\wt{\SU(n,1)}$}\label{ssec:SUhat}

We now want to study the universal cover $\pi : \wt{G} \to G$, which fits into a central extension:
\[
1 \lra \bbZ \lra \wt{G} \overset{\pi}{\lra} G \lra 1
\]
where we identify $\bbZ$ with $\pi_1(G)$. First, we make the following observation, which follows immediately from the identification of the center of $G$ with $\bbZ / (n+1)$.

\begin{lem}\label{lem:GhatCenter}
If $Z(\wt{G})$ denotes the center of $\wt{G}$, then $Z(\wt{G}) \cong \bbZ$, $\pi_1(G)$ is naturally identified with $(n+1) Z(\wt{G})$, and $Z(\wt{G})$ is naturally $\pi_1(\PU(n,1))$.
\end{lem}

We now lift our Iwasawa decomposition of $G$ to one for $\wt{G}$.

\begin{lem}\label{lem:LiftIwasawa}
With notation as above, let $\wt{K}$ be the universal cover of $K$. Then we have embeddings of universal covers $\wt{B}_0 = B_0, \wt{K} \to \wt{G}$ so that $B_0 \to \wt{G}$ lifts $B_0 < G$, $\pi|_{\wt{K}}$ is the universal cover $\wt{K} \to K$, and $B_0 \wt{K}$ is an Iwasawa decomposition of $\wt{G}$.
\end{lem}

\begin{pf}
This follows from the fact that $\pi$ induces an isomorphism between the Lie algebras of $\wt{G}$ and $G$. Continuing with the notation from the proof of Lemma~\ref{lem:Iwasawa}, $\Theta$ determines an Iwasawa decomposition of $\wt{G}$ with
\[
\wt{A}_+ \wt{N} = \exp_{\wt{G}}(\fraka \frakn),
\]
where $\frakn \subset \frakg$ is the subalgebra associated with $N$. Since $B_0$ is contractible and $\exp_G : \fraka \frakn \to B_0$ is a diffeomorphism, $\exp_{\wt{G}}(\fraka \frakn) \cong B_0$ is a lift of $B_0 < G$.

Set $\wt{K} = \exp_{\wt{G}}(\frakk)$, so we have an Iwasawa decomposition $B_0 \wt{K}$ of $\wt{G}$. It remains to show that $\wt{K}$ is the universal cover of $K$ with universal covering map $\pi$. Since $B_0$ is contractible, we have $\pi$-equivariant diffeomorphisms $B_0 \bs \wt{G} \to \wt{K}$ and $B_0 \bs G \to K$ inducing homotopy equivalences of $\wt{G}$ with $\wt{K}$ and $G$ with $K$. Since $\wt{K}$ is then connected and simply connected, $\pi|_{\wt{K}}$ must be the universal cover. This proves the lemma.
\end{pf}

Applying the identification of $\wt{K}$ with $\SU(n) \times \bbR$, we have the following.

\begin{cor}\label{cor:HatIwasawa}
There is an Iwasawa decomposition $\wt{G} = B_0 (\SU(n) \times \bbR)$ compatible with the Iwasawa decomposition $G = B_0 K$ and the universal cover $\SU(n) \times \bbR \to K$.
\end{cor}

\subsection{Useful homogeneous spaces}\label{ssec:HomogSp}

We now study a pair of homogeneous spaces for $G$ and $\wt{G}$ that will end up being very helpful in computing lifts of relations in $G$ up to $\wt{G}$. Specifically, the homogeneous space $G / \SU(n)$ is useful in calculating the element of $\pi_1(G) \cong \bbZ$ associated with a loop $\sig : S^1 \to G$ using only linear algebra.

\begin{lem}\label{lem:HomogSp}
Identifying $G / K$ with $\bbB^n$ induces a diffeomorphism between $G / \SU(n)$ and $\bbB^n \times \U(1)$. The map $g \SU(n) \mapsto g z_0$ moreover defines a left $G$-invariant diffeomorphism between $G / \SU(n)$ and the space of all vectors in $\bbC^{n+1}$ of $h$-norm $-2$.
\end{lem}

\begin{pf}
Identifying $G / K$ with $\bbB^n$, the Iwasawa decomposition $G = B_0 K$ identifies $B_0$ with $\bbB^n$. Then $k = (h, \xi) \mapsto \xi$ identifies $K / \SU(n)$ with $\U(1)$, hence mapping $b (h, \xi) \in B_0 K$ to $(b, \xi)$ defines a diffeomorphism between $G / \SU(n)$ and $\bbB^n \times \U(1)$. This proves the first claim. The second claim is the orbit-stabilizer theorem, since $\SU(n) < K$ is the stabilizer of $z_0$ and $G$ acts transitively on vectors of $h$-norm $-2$.
\end{pf}

Corollary~\ref{cor:HatIwasawa} then gives a diagram of universal covers
\[
\begin{tikzcd}
\wt{G} / \SU(n) \arrow{r}{\simeq} \arrow{d} & \bbB^n \times \bbR \arrow{d} \\
G / \SU(n) \arrow{r}{\simeq} & \bbB^n \times \U(1)
\end{tikzcd}
\]
and the long exact sequence of homotopy groups also implies that the fibration
\[
\begin{tikzcd}
G / \SU(n) \arrow{r} & \bbB^n \times \U(1) \arrow{d} \\ & \U(1)
\end{tikzcd}
\]
induces an isomorphism on the level of fundamental group. We can then compute elements of $\pi_1(G)$ as follows.

\begin{lem}\label{lem:Computepi1}
Given a loop $\sig : [0,1] \to G$ based at the identity of $G$, we obtain a loop $\wh{\sig}$ based at $1 \in \bbC$ by
\[
\begin{tikzcd}
{[0,1]} \arrow{r}{\sig} \arrow[swap]{drr}{\wh{\sig}} & G \arrow{r} & G z_0 \arrow{d}{\pi_{n+1}} \\
 & &  \bbC
\end{tikzcd}
\]
where $\pi_{n+1}$ is projection of $\bbC^{n+1} \supset G z_0$ onto the $(n+1)^{st}$ coordinate. Then $\wh{\sig}([0,1]) \subset \bbC \ssm \{0\}$. Identifying $\pi_1(G)$ with $\bbZ$, the element $[\sig] \in \pi_1(G)$ associated with $\sig$ is equal to the winding number of $\wh{\sig}$ around $0 \in \bbC$.
\end{lem}

\begin{pf}
For $s \in [0,1]$, let
\[
\sig(s) = b_s (g_s, \xi_s) \in B_0 K
\]
be the Iwasawa decomposition of $\sig(s)$. We then see that
\[
\sig(s) z_0 = \begin{pmatrix} \textasteriskcentered \\ \vdots \\ \textasteriskcentered \\ \lam_s^{-1} \xi_s \end{pmatrix},
\]
where $\lam_s^{\pm 1}$ are the diagonal entries of $b_s$ with respect to the coordinates on $B_0$ defined in \S\ref{ssec:SUcoords}, hence
\[
\pi_{n+1}(\sig(s) z_0) = \lam_s^{-1} \xi_s \in \bbC \ssm \{0\}.
\]
Since $\sig$ is based at the identity of $G$, $\lam_0 = \xi_0 = 1$, so $\wt{\sig}$ is based at $1 \in \bbC$. This proves that $\wh{\sig}([0,1])$ is a loop in $\bbC \ssm \{0\}$ based at $1$.

Recall that the identification of $G z_0$ with $\bbB^n \times \U(1)$ in Lemma~\ref{lem:HomogSp} identifies the $\U(1)$ factor with $\U(1) z_0$ for $\U(1) < K$. This, with the previous calculation of $\pi_{n+1}$ on $G z_0$, implies that the map
\[
\begin{tikzcd}
G \arrow{r} \arrow[dashed]{dr} & G z_0 \arrow{d}{\pi_{n+1}} \\
& \bbC \ssm \{0\}
\end{tikzcd}
\]
induces an isomorphism $\pi_1(G) \cong \pi_1(\bbC \ssm \{0\})$. The element of $\pi_1(\bbC \ssm \{0\})$ associated with a loop is computed by the winding number around $0$, so this completes the proof of the lemma.
\end{pf}

\section{Lifting presentations}\label{sec:Lift}

In this section, fix a finitely presented group $\Gam < G = \SU(n,1)$. We assume for simplicity that $\Gam$ contains the center $Z(G) \cong \bbZ / (n+1)$ of $G$, which is always the case when $\Gam$ is the pullback to $\SU(n,1)$ of a lattice in $\PU(n,1)$, and let $\wh{z}$ be a generator for $Z(G)$. Without loss of generality, we can assume $\wh{z}$ is a generator for $\Gam$, and so $\Gam$ is given by generators and relations
\[
\Gam = \langle g_1, \dots, g_d, \wh{z}~|~\calR_1,\dots,\calR_e \rangle.
\]
We now describe a procedure for presenting the central extension
\[
1 \lra \bbZ \lra \wt{\Gam} \lra \Gam \lra 1
\]
obtained by pulling $\Gam$ back to the universal cover $\wt{G}$ of $G$. 

\medskip

\noindent
\textbf{Lifting generators}

\medskip

Since $G$ is connected, we can write $g_i = \exp_G(v_i)$ for some $v_i$ in the Lie algebra $\frakg$ of $G$. This defines an embedded path:
\begin{align*}
\gam_i : [0,1] &\lra G & \gam_i(s) &= \exp_G(s v_i)
\end{align*}
Notice that $g_i^{-1} = \exp_G(-v_i)$. Then
\begin{align*}
\wt{\gam}_i : [0,1] &\lra \wt{G} & \wt{\gam}_i(s) &= \exp_{\wt{G}}(s v_i)
\end{align*}
is an embedded path in $\wt{G}$ with $\wt{g}_i = \wt{\gam}_i(1)$ a lift of $g_i$ and $\wt{g}_i^{-1} = \exp_{\wt{G}}(-v_i)$.

Let $z$ be a generator for the center of $\wt{G}$ that projects to $\wh{z} \in G$. Then $z^{n+1}$ is a generator for $\pi_1(G)$. Our generators for $\wt{\Gam}$ are $\wt{g}_1, \dots, \wt{g}_d$ and $z$.

\medskip

\noindent
\textbf{Lifting relations}

\medskip

Now, let $\calR_i$ be a relation for $\Gam$. There is no loss of generality in considering $\calR_i$ as being of the form
\[
g_{i_f}^{\ep_{1,f}} \cdots g_{i_1}^{\ep_{i,1}} = I_{n+1} \in G
\]
for $\ep_{i,j} \in \{\pm 1\}$ and $f = f(i)$. This determines a piecewise-defined loop $\sig_i : [0,1] \to G$ based at the identity by
\[
\sig_i(s) = \begin{cases}
\gam_{i_1}(\ep_{i_1} f s) & s \in [0, \frac{1}{f}] \\
\gam_{i_j}(\ep_{i,j} (f s - j + 1)) g_{i_{j-1}}^{\ep_{i,j-1}} \cdots g_{i_1}^{\ep_{i,1}} & s \in [\frac{j-1}{f}, \frac{j}{f}] \\
\gam_{i_f}(\ep_{i,f} (f s - f + 1)) g_{i_{f-1}}^{\ep_{i,f-1}} \cdots g_{i_1}^{\ep_{i,1}} & s \in [\frac{f-1}{f}, 1]
\end{cases}
\]
where $1 \le j \le f-1$. Using Lemma~\ref{lem:Computepi1}, we can compute the element $r_i = [\sig_i] \in \bbZ \cong \pi_1(G)$ defined by $\sig_i$ using the winding number of the loop $\pi_{n+1}(\sig_i(s) z_0)$ around $0$ in $\bbC \ssm \{0\}$. Notice that $\sig_1(s) \in \SL_{n+1}(\bbC)$, so the path $\sig_i(s) z_0$ is explicitly computed using matrix multiplication in $\bbC^{n+1}$.

Uniqueness of path lifting implies that the lift of $\sig_i$ to $\wt{G}$ based at the identity is given by $z^{(n+1) r_i}$ for $r_i \in \bbZ$ the winding number of $\pi_{n+1} \circ \sig_i$ around $0 \in \bbC$. Thus our lifted relations $\wt{\calR}_i$ for $\wt{\Gam}$ are:
\[
\wt{g}_{i_f}^{\ep_{1,f}} \cdots \wt{g}_{i_1}^{\ep_{i,1}} \wt{z}^{-(n+1) r_i} = 1_{\wt{G}}.
\]
This presents $\wt{\Gam}$ as the group with generators $\{\wt{g}_i\}$ and relations $\{\wt{\calR}_i\}$, possibly with the additional relation that $z$ is central.

\begin{lem}\label{lem:CorrectPresentation}
With notation as above,
\[
\wt{\Gam} = \langle \wt{g}_1, \dots, \wt{g}_d, z~|~\wt{\calR}_1, \dots, \wt{\calR}_e, z~\mathrm{central} \rangle.
\]
\end{lem}

\begin{pf}
Let $L$ be the abstract group with the given presentation. The above shows that each relation for $L$ holds in $\wt{\Gam}$, hence we have a homomorphism $L \to \wt{\Gam}$. On the other hand, $\langle z \rangle$ is the kernel of the projection $L \to \Gam$, which factors through $\wt{\Gam}$. Thus, the kernel of $L \to \wt{\Gam}$ is contained in $\langle z \rangle \cong \bbZ$. However, $\langle z \rangle$ maps isomorphically to the center of $\wt{\Gam}$, so $L \cong \wt{\Gam}$.
\end{pf}

\section{Residual finiteness of central extensions and the cup product}\label{sec:Beauville}

Let $\Gam$ be the fundamental group of a Riemann surface $\Sig$ of genus $g \ge 2$. The fundamental group $\wt{\Gam}$ of its unit tangent bundle $T^1 \Sig$ fits into a central exact sequence
\[
1 \lra \bbZ \lra \wt{\Gam} \lra \Gam \lra 1
\]
whose Euler class $e \in H^2(\Sig, \bbZ) \cong H^2(\Gam, \bbZ)$ is the first Chern class of the canonical bundle. We can equivalently say that the Euler class of the bundle is a nonzero multiple of the K\"ahler form $\om$ on $\Sig$. Lastly, $\wt{\Gam}$ is isomorphic to the preimage of $\Gam < \PSL_2(\bbR)$ in its universal cover $\wt{\PSL}_2(\bbR)$.

It is well-known that $\wt{\Gam}$ is residually finite, and the most common proof is to show that $\wt{\Gam}$ admits a homomorphism onto an integral Heisenberg group such that a generator for the center of $\wt{\Gam}$ maps to an element of infinite order. See \cite[\S IV.48]{delaHarpe} for details. It is not hard to prove by hand using reductions modulo primes, that integral Heisenberg groups are residually finite, and recall from \S\ref{sec:RF} that in fact all finitely generated nilpotent groups are residually finite. Since $\Gam$ is linear, and hence residually finite, residual finiteness of $\wt{\Gam}$ follows from Corollary~\ref{cor:RFNilQuo}.

The purpose of this section is to present a broad generalization of this result, and its strategy of proof, based on a proposition that is perhaps best attributed to Sullivan (e.g., see \cite[Ch.\ 3, \S1]{ABCKT}, in particular Remark 3.4). In what follows it will be convenient to adopt the convention that a $k$-step nilpotent group has length \emph{at most} $k$. Our most general result is the following.

\begin{thm}\label{thm:CentralRF1}
Let $X$ be a closed aspherical manifold with fundamental group $\Gam$ and
\[
U(1) \lra Y \lra X
\]
be a principal $\U(1)$ bundle with Euler class $\om \in H^2(X, \bbZ)$ that has infinite additive order, equivalently, its image $\om\in H^2(X,\bbQ)$ is not zero. If $\wt{\Gam} = \pi_1(Y)$ and $z$ denotes a generator for $\pi_1(U(1)) < \wt{\Gam}$, then the image of $z$ in the maximal two-step nilpotent quotient of $\wt{\Gam}$ has infinite order if and only if $\om$ is in the image of the map
\[
c_\bbQ : \bigwedge\nolimits^2 H^1(X, \bbQ) \to H^2(X, \bbQ)
\]
given by evaluation of the cup product.
\end{thm}

\begin{rem}
The aspherical assumption on $X$ is only to conclude using the standard homotopy exact sequence that $\wt{\Gam}$ is indeed a central extension of $\Gam$ by $\bbZ$. Our methods also allow one to study the cases where any of our hypotheses on $X$ and $Y$ fail to hold, but we do not pursue these directions since this would take us too far afield from our applications to ball quotients.
\end{rem}

Before embarking on the proof, we record some corollaries. The first explains why Theorem~\ref{thm:CentralRF1} is exactly what we need in order to implement the proof of residual finiteness of $\pi_1(T^1\Sig)$ for more general circle bundles over aspherical manifolds.

\begin{cor}\label{cor:CentralRF2}
With $X$, $Y$, $\Gam$, and $\wt{\Gam}$ as in the statement of Theorem~\ref{thm:CentralRF1}, suppose that $\Gam$ is residually finite. Then residual finiteness of $\wt{\Gam}$ can be detected in a two-step nilpotent quotient of $\wt{\Gam}$ via Corollary~\ref{cor:RFNilQuo} if and only if the Euler class of $Y \to X$ is in the image of the cup product map $c_\bbQ$.
\end{cor}

\begin{pf}
Since $\Gam$ is a finitely generated residually finite group, Corollary~\ref{cor:RFNilQuo} says that it suffices find a nilpotent quotient of $\wt{\Gam}$ where the generator $z$ for $\pi_1(\U(1))$ maps to an element of infinite order. Taking the maximal two-step quotient of $\wt{\Gam}$, the corollary then follows from Theorem~\ref{thm:CentralRF1}.
\end{pf}

The following is then a consequence of Theorem~\ref{thm:CentralRF1} and Corollary~\ref{cor:RFNilQuo}.

\begin{cor}\label{cor:RFCentralKahler}
Let $X$ be an aspherical compact K\"ahler manifold with residually finite fundamental group for which the K\"ahler class $\om \in H^2(X)$ is in the image of the cup product map $\bigwedge\nolimits^2 H^1(X, \bbQ) \to H^2(X, \bbQ)$. If $Y$ is a principal $\U(1)$ bundle over $X$ with Euler class a nonzero multiple of $\om$, then $\pi_1(Y)$ is residually finite.
\end{cor}

\begin{pf}
The case where $X$ is a Riemann surface of genus $g \ge 2$ is the classical case explained at the beginning of this section. When $X$ is a torus, $\pi_1(Y)$ is a finitely generated nilpotent group, so again the result is known. If $X$ has complex dimension $n \ge 2$, then $\om^n \in H^{2n}(X, \bbQ)$ is nontrivial, hence $\om$ has infinite order in $H^2(X, \bbQ)$. Therefore, Corollary~\ref{cor:CentralRF2} applies to prove that $\pi_1(Y)$ is residually finite.
\end{pf}

Note that Theorem~\ref{thm:GeneralCup} is the special case of Corollary~\ref{cor:RFCentralKahler} where $X$ is a smooth compact ball quotient. Also recall that smooth compact ball quotients are smooth projective varieties, so the real cohomology class of the K\"ahler form is always (up to normalization, of course) rational.

Before proving Theorem~\ref{thm:CentralRF1}, we define some notation and give some preliminary results. To start, we only assume that $X$ is a closed manifold with fundamental group $\Gam$. Let $\Del$ denote the derived subgroup of $\Gam$. The earliest reference we could find to the following fact was in work of Sullivan \cite{Sullivan}, who commented that its proof follows from ``a certain amount of soul searching classical algebraic topology''. Also see \cite[Ch.\ 3]{ABCKT}. We use a formulation given by Beauville \cite{BeauvilleDerived}.

\begin{prop}[Cor.\ 1(2) \cite{BeauvilleDerived}]\label{prop:Beauville}
Let $X$ be a connected space homotopic to a CW complex with fundamental group $\Gam$ and $\Del$ be the derived subgroup of $\Gam$. Then there is a canonical exact sequence
\begin{equation}\label{eq:Beauville}
1 \lra \Hom(\Del / [\Del, \Gam], \bbQ) \lra \bigwedge\nolimits^2 H^1(X, \bbQ) \overset{c_\bbQ}{\lra} H^2(X, \bbQ)
\end{equation}
with $c_\bbQ$ the cup product map.
\end{prop}

We now relate the cup products on $X$ and $\U(1)$ bundles over $X$ via the Gysin sequence.

\begin{lem}\label{lem:BundleCohomology}
Let $X$ and $Y$ be as in the statement of Theorem~\ref{thm:CentralRF1}, and suppose that $\om \in H^2(X, \bbZ)$ has infinite additive order. Then the class of the fiber in $H_1(Y, \bbQ)$ is trivial, $H^i(Y, \bbQ) \cong H^i(X, \bbQ)$ for $i = 0,1$, and pullback under the projection $\pi : Y \to X$ induces an injection of $H^2(X, \bbQ)/\langle \om \rangle$ into $H^2(Y,\bbQ)$. Moreover, the diagram
\begin{equation}\label{eq:WedgeMap}
\begin{tikzcd}[column sep=small,scale cd=0.95]
1 \arrow{r} & \Hom(\Del / [\Del, \Gam], \bbQ) \arrow{r} \arrow{d}{\pi^*} & \bigwedge\nolimits^2 H^1(X, \bbQ) \arrow{r}{c_\bbQ^X} \arrow{d}{\pi^*} & H^2(X, \bbQ) \arrow{dr} \arrow{d}{\pi^*} & \\
1 \arrow{r} & \Hom(\wt{\Del} / [\wt{\Del}, \wt{\Gam}], \bbQ) \arrow{r} & \bigwedge\nolimits^2 H^1(Y, \bbQ) \arrow{r}{c_\bbQ^Y} & H^2(Y, \bbQ) \arrow[hookleftarrow]{r} & H^2(X, \bbQ)/\langle \om \rangle
\end{tikzcd}
\end{equation}
commutes, where $c_\bbQ^X, c_\bbQ^Y$ are the cup product maps $\Gam$, $\wt{\Gam}$ are the fundamental groups of $X$, $Y$, and $\Del$, $\wt{\Del}$ are the derived subgroups of $\Gam$, $\wt{\Gam}$.
\end{lem}

\begin{pf}
Throughout this proof, all homology groups have $\bbQ$ coefficients. Let $w_i : H^i(X) \to H^{i+2}(X)$ be the map induced by cup product with $\om$. Then the standard Gysin sequence for $Y$ \cite[Thm.\ 5.7.11]{Spanier} gives
\[
\begin{tikzcd}
0 \arrow{r} & H^1(X) \arrow{r}{\pi^*} \arrow[d,phantom, ""{coordinate, name=Z}] & H^1(Y) \arrow{r}{\pi_*} & H^0(X) \arrow{r}{w_0} & H^2(X) \arrow[dllll,rounded corners,to path={ --([xshift=2ex]\tikztostart.east)|- (Z)[near end]\tikztonodes-| ([xshift=-2ex]\tikztotarget.west)-- (\tikztotarget)}, "\pi^*" above] \\
H^2(Y) \arrow{r}{\pi_*} & H^1(X) \arrow{r}{w_1} & H^3(X) \arrow{r}{\pi^*} & \cdots
\end{tikzcd}
\]
where $\pi^*$ is pullback and $\pi_*$ is integration over the fiber. Then $\om$ has infinite order in $H^2(X, \bbZ)$ if and only if it is nonzero in $H^2(X)$, which is moreover true if and only if $w_0$ is injective. We conclude that $\pi^*$ induces an isomorphism between $H^i(X)$ and $H^i(Y)$ for $i = 0,1$ when $\om$ has infinite order in $H^2(X, \bbZ)$. Therefore $H^1(Y) \to H^0(X)$ is the zero map, and since this map is integration over the fiber, the class of the fiber in $H_1(Y)$ must be trivial.

Further, if $\om$ has infinite order in $H^2(X)$, then $\pi^*$ induces an injection $H^2(X) / \langle \om \rangle \hookrightarrow H^2(Y)$. Commutativity of the diagram in the statement of the lemma now follows from Proposition~\ref{prop:Beauville} and the fact that pullback is a ring homomorphism for the cup product. This completes the proof.
\end{pf}

We are now prepared for the main technical result on two-step nilpotent quotients that we need for the proof of Theorem~\ref{thm:CentralRF1}.

\begin{lem}\label{lem:NilQuo}
Suppose that $X$, $Y$, $\Gam$, and $\wt{\Gam}$ satisfy the hypotheses of Theorem~\ref{thm:CentralRF1}. Let $\Del$ (resp.\ $\wt{\Del}$) be the derived subgroup of $\Gam$ (resp.\ $\wt{\Gam}$), and let $z$ be a generator for $\ker(\wt{\Gam} \to \Gam)$. Then
\begin{equation}\label{eq:NilQuo}
\mathrm{dim}(\Hom(\wt{\Del} / [\wt{\Del}, \wt{\Gam}], \bbQ)) = \mathrm{dim}(\Hom(\Del / [\Del, \Gam], \bbQ)) + \ep
\end{equation}
for $\ep \in \{0,1\}$. Moreover, $\ep = 1$ if and only if the image of $z$ in the maximal two-step nilpotent quotient of $\wt{\Gam}$ has infinite order.
\end{lem}

\begin{pf}
Let
\begin{align*}
N &= \Gam / [\Del, \Gam] & \wt{N} &= \wt{\Gam} / [\wt{\Del}, \wt{\Gam}]
\end{align*}
be the maximal two-step nilpotent quotients of $\Gam$ and $\wt{\Gam}$. The derived subgroups of $N$ and $\wt{N}$ are $\Del / [\Del, \Gam]$ and $\wt{\Del} / [\wt{\Del}, \wt{\Gam}]$, respectively. Let $\wh{z}$ denote the image of $z$ in $\wt{N}$.

Since $z$ generates the fundamental group of the fiber, and the fiber is trivial in $H_1(Y, \bbQ)$ by Lemma~\ref{lem:BundleCohomology}, the image of $z$ in the abelianization of $\wt{\Gam}$ must have finite order. Therefore, we conclude that there is some $m \ge 1$ so that $\wh{z}{\,}^m \in \wt{\Del} / [\wt{\Del}, \wt{\Gam}]$. Indeed, $\wt{N}$ and $\wt{\Gam}$ have the same abelianization and $\wt{\Del} / [\wt{\Del}, \wt{\Gam}]$ is the commutator subgroup of $\wt{N}$. We then have a diagram:
\[
\begin{tikzcd}
 & \textcolor{blue}{1} \arrow[blue]{d} & 1 \arrow{d} & 1 \arrow{d} & \\
1 \arrow{r} & \textcolor{blue}{\langle\, \wh{z}{\,}^m\, \rangle} \arrow{r} \arrow[blue]{d} & \langle\, \wh{z}\, \rangle \arrow{r} \arrow{d} & \bbZ / m \arrow{r} \arrow{d} & 1 \\
1 \arrow{r} & \textcolor{blue}{\wt{\Del} / [\wt{\Del}, \wt{\Gam}]} \arrow{r} \arrow[blue]{d} & \wt{N} \arrow{r} \arrow{d} & \wt{\Gam} / \wt{\Del} \arrow{r} \arrow{d} & 1 \\
1 \arrow{r} & \textcolor{blue}{\Del / [\Del, \Gam]} \arrow{r} \arrow[blue]{d} & N \arrow{r} \arrow{d} & \Gam / \Del \arrow{r} \arrow{d} & 1 \\
& \textcolor{blue}{1} & 1 & 1 &
\end{tikzcd}
\]
To justify this, one wants to know that $\wt{N} / \langle\, \wh{z}\, \rangle = N$. Note that $\wt{\Gam} \to \wt{N} / \langle\, \wh{z}\, \rangle$ factors through $\wt{\Gam} / \langle z \rangle = \Gam$, so $\wt{N} / \langle\, \wh{z}\, \rangle$ is a quotient of $\Gam$. Moreover, $\wt{\Del}$ projects to $\Del$ and $[\wt{\Del}, \wt{\Gam}]$ has image $[\Del, \Gam]$. Thus the kernel of $\Gam \to \wt{N} / \langle\, \wh{z}\, \rangle$ is $[\Gam, \Del]$, i.e., the kernel of $\Gam \to N$, hence $N = \wt{N} / \langle\, \wh{z}\, \rangle$.

Then the vertical exact sequence in \textcolor{blue}{blue} is an exact sequence of finitely generated abelian groups. Moreover, the free ranks of $\wt{\Gam} / [\wt{\Del}, \wt{\Gam}]$ and $\Gam / [\Del, \Gam]$ differ by $\ep \in \{0,1\}$ with $\ep = 1$ if and only if $\wh{z}{\, }^m$ (equivalently, $\wh{z}$) has infinite order in $\wt{N}$. These are equivalent to Equation~\eqref{eq:NilQuo} by taking $\Hom(-, \bbQ)$, so this proves the lemma.
\end{pf}

We are now prepared to prove Theorem~\ref{thm:CentralRF1}.

\begin{pf}[Proof of Theorem~\ref{thm:CentralRF1}]
Let $K_X$, $I_X$, $K_Y$, $I_Y$ denote the kernel and image of $c_\bbQ^X$ and $c_\bbQ^Y$, and $k_X$, $i_X$, $k_Y$, $i_Y$ denote their dimensions over $\bbQ$. Pullback induces an isomorphism between $H^1(X, \bbQ)$ and $H^1(Y, \bbQ)$ by Lemma~\ref{lem:BundleCohomology}, and see from Equation~\eqref{eq:WedgeMap} that $\pi^*(I_X) = I_Y$. Therefore we have equalities:
\begin{align*}
k_X + i_X &= \dim_\bbQ \bigwedge\nolimits^2 H^1(X, \bbQ) \\
&= \dim_\bbQ \bigwedge\nolimits^2 H^1(Y, \bbQ) \\
&= k_Y + i_Y
\end{align*}
In other words,
\begin{equation}\label{eq:DimEq}
i_X - i_Y = k_Y - k_X.
\end{equation}
Moreover, the proof of Lemma~\ref{lem:BundleCohomology} implies that $i_X = i_Y$ if $\om \notin I_X$ and $i_X = i_Y + 1$ if $\om \in I_X$. We conclude that the left side of Equation~\eqref{eq:DimEq} equals $1$ if and only if $\om \in I_X$. Similarly, Equation~\eqref{eq:WedgeMap} and Lemma~\ref{lem:NilQuo} imply that the right side of Equation~\eqref{eq:DimEq} is $1$ if and only if the image of $z$ in the maximal two-step nilpotent quotient of $\wt{\Gam}$ has infinite order. This proves the theorem.
\end{pf}

We now give an example where Corollary~\ref{cor:RFCentralKahler} applies to prove residual finiteness of a lattice in the universal cover of $\PU(2,1)$. Let $\Gam_S < \PU(2,1)$ be the fundamental group of the Stover surface $X_S$, a smooth compact ball quotient surface first studied in \cite{StoverHurwitz} and explored further by Dzambic and Roulleau \cite{DzambicRoulleau}. In particular, Dzambic and Roulleau proved that the image of the cup product $\bigwedge\nolimits^2 H^1(X_S, \bbQ) \to H^2(X_S, \bbQ)$ contains $H^{1,1}(X_S, \bbQ)$ \cite[Thm.~5]{DzambicRoulleau}, hence the canonical class is in the image of the cup product map. Corollary~\ref{cor:RFCentralKahler} then implies that the preimage of $\Gam_S$ in the universal cover of $\PU(2,1)$, which is the central extension with Euler class the canonical class, is residually finite. Since $\Gam_S$ is commensurable with both the Cartwright--Steger surface and the Deligne--Mostow lattice with weights $(11,7,2,2,2)/12$ (see \cite{StoverHurwitz}), Lemma~\ref{lem:CommensurableRF} allows us to conclude:

\begin{thm}\label{thm:StoverRF}
Let $\Gam < \PU(2,1)$ be either the Deligne--Mostow lattice with weights $(11,7,2,2,2)/12$, the fundamental group of the Cartwright--Steger surface, the fundamental group of the Stover surface, or any other lattice in $\PU(2,1)$ commensurable with these. Then the preimage $\wt{\Gam}$ of $\Gam$ in the universal cover of $\PU(2,1)$ is residually finite.
\end{thm}

\begin{rem}
We cannot use Corollary~\ref{cor:RFCentralKahler} to prove Theorem~\ref{thm:StoverRF} for the Cartwright--Steger surface $X_{CS}$ directly. Since $h^{1,0}(X_{CS}) = 1$, the image of the cup product from $H^1$ to $H^2$ is the one-dimensional subspace spanned by $\de = \al \wedge \conj{\al}$ for some $\al \in H^{1,0}(X_{CS})$. Then $\de \wedge \de = 0$, hence $\de$ cannot be a multiple of the canonical class and Corollary~\ref{cor:RFCentralKahler} does not apply. One can show with Magma (e.g., exactly as in \cite[\S 2]{DzambicRoulleau}) that the derived subgroup of the maximal two-step nilpotent quotient of the preimage of $\pi_1(X_{CS})$ in the universal cover is finite.
\end{rem}

\begin{rem}
Since, as described above, every finitely generated nilpotent group has a torsion-free subgroup of finite index, and torsion-free nilpotent groups are linear, it isn't hard to see that the same holds for an arbitrary finitely generated nilpotent group using the induced representation. More generally, linearity is a commensurability invariant of groups. Therefore, if we assume instead throughout this section that $\Gam = \pi_1(X)$ is linear, then we conclude that $\wt{\Gam} = \pi_1(Y)$ is also linear. Indeed, with notation as above, if $\Gam$ embeds in $\GL_{N_1}(\bbR)$ and $\wt{N}$ into $\GL_{N_2}(\bbR)$, the natural composition
\[
\wt{\Gam} \to \Gam \times \wt{N} \hookrightarrow \GL_{N_1+N_2}(\bbR)
\]
is a faithful linear representation of $\wt{\Gam}$. Thus the results in this section hold with residual finiteness replaced with linearity.
\end{rem}

\section{Complex projective surfaces with the same fundamental group as a ball quotient}\label{sec:Samepi1}

The purpose of this section is to prove the following result.

\begin{thm}\label{thm:pi2=0}
Let $M$ be a smooth complex projective surface such that $\pi_2(M) = \{0\}$ and $\pi_1(M)$ is isomorphic to a torsion-free cocompact lattice $\Gam$ in $\PU(2,1)$. Then $M$ is biholomorphic to $\Gam \bs \bbB^2$.
\end{thm}

We start with some basic facts that we will need in the proof.

\begin{lem}\label{lem:HolMap}
If $M$ is a smooth complex projective surface such that $\pi_1(M)$ is isomorphic to a torsion-free cocompact lattice $\Gam < \PU(2,1)$, then there is a surjective holomorphic map $f : M \to \Gam \bs \bbB^2$ realizing the isomorphism $\pi_1(M) \cong \Gam$.
\end{lem}

\begin{pf}
Set $N = \Gam \bs \bbB^2$. Then $N$ is a $K(\Gam, 1)$ that realizes $\Gam$ as a Poincar\'e duality group of dimension four and there is a nonconstant continuous map $f : M \to N$ realizing this isomorphism on fundamental groups. Using standard arguments from the theory of harmonic maps and Siu Rigidity \cite[Thm.~1]{Siu}, we can assume that $f$ is harmonic, hence pluriharmonic. Specifically, it follows from \cite[Thm.~7.2(b)]{Carlson-ToledoIHES} that there are three possibilities:
\begin{enum}

\item $f(M)$ is a closed geodesic on $N$;

\item there is a Riemann surface $C$ so that $f$ factors as the composition
\[
\begin{tikzcd}
M \arrow{r}{\phi} \arrow{dr}{f} & C \arrow{d}{\psi} \\
 & N
\end{tikzcd}
\]
where $\phi : M \to C$ is a surjective holomorphic map and $\psi : C \to N$ is harmonic;

\item $f$ is surjective and holomorphic (or conjugate holomorphic, which can be safely ignored by a change of complex structure).

\end{enum}
However, the first two cases are impossible. Indeed, in each case the isomorphism $f_* : \Gam \overset{\sim}{\lra} \Gam$ would factor through a surjective homomorphism onto a group of cohomological dimension at most two, which is absurd since $\Gam$ has cohomological dimension four. This proves the lemma.
\end{pf}

\begin{prop}\label{prop:HConvex}
Suppose that $M$ is a smooth complex projective surface with $\pi_1(M)$ isomorphic to a torsion-free cocompact lattice $\Gam < \PU(2,1)$. Then the universal cover $\wt{M}$ of $M$ is holomorphically convex.
\end{prop}

\begin{pf}
This follows directly from work of Katzarkov and Ramachandran \cite[Thm.~1.2]{Katzarkov-Ramachandran}, since $\Gam$ is one-ended and the natural inclusion of $\Gam$ into $\PGL_3(\bbC)$ is Zariski dense. More generally, the Shafarevich conjecture is known for smooth projective varieties with linear fundamental group by more recent work of Eyssidieux, Katzarkov, Pantev, and Ramachandran \cite{EKPR}.
\end{pf}

\begin{rem}\label{rem:pinf}
We sketch an elementary proof of Proposition~\ref{prop:HConvex}. Suppose that $f : M \to N$ is the holomorphic map from Lemma~\ref{lem:HolMap}.  Consider the diagram
\[
\begin{tikzcd}
\wt{M} \arrow{r}{\wt{f}} \arrow{d}[left]{\pi_M} & \bbB^2 \arrow{d}{\pi_N} \\
M \arrow{r}[below]{f} & N
\end{tikzcd}
\]
where $\wt{M}$ is the universal cover of $M$ and $\wt{f}$ is a $f_*$-equivariant lift of $f$ to the universal covers, where $f_*:\pi_1(M)\to \Gamma$ is the induced homomorphism and $\pi(M), \Gamma$ act by covering transformations. Since $\bbB^2$ is holomorphically convex, to prove that $\wt{M}$ is holomorphically convex it suffices to show that $\wt{f}$ is proper. However, properness of $\wt{f}$ is an easy consequence of the fact that $f_*$ is an isomorphism and compactness of $M, N$. In more detail, using these facts it is easy to show:

\medskip

\noindent
\textbf{Key claim:} For every $y_0 \in N$ and $\wt{y}_0 \in \pi_N^{-1}(y_0) \subset \bbB^2$, the map
\[
\pi_M : \wt{f}^{-1}(\wt{y}_0) \lra f^{-1}(y_0)
\]
is a homeomorphism.

\medskip\noindent
This shows that $\wt{f}$-preimages of points are compact. A slight variation of the argument gives that $\wt{f}$-preimages of sufficiently small closed disks are compact, from which properness follows easily.
\end{rem}

\begin{pf}[Proof of Theorem~\ref{thm:pi2=0}]
With the notation established earlier in this section, since $\wt{M}$ is simply connected and $\pi_2(\wt{M}) = 0$, standard topological arguments show that $M$ is aspherical if and only if $H_3(\wt{M},\bbZ ) = 0$. Consider the Cartan--Remmert reduction $p:\wt{M}\to Y$ of the holomorphically convex space $\wt{M}$. \emph{A priori}, the Stein space $Y$ could be one- or two-dimensional. Since $\wt{f}$ is surjective and factors through $p$, the composition
\begin{equation}\label{eq:CartRem}
 \wt{M}\to Y \to \bbB^2
\end{equation}
implies that $Y$ admits a surjective map to $\bbB^2$, hence $Y$ is two-dimensional.

This means that $p:\wt{M}\to Y$ contracts a collection of disjoint compact, connected complex curves in $\wt{M}$ (the maximal compact, connected subvarieties of $\wt{M}$) to a discrete set of points in $Y$. From this it follows that $H_i(\wt{M},\bbZ) \cong H_i(Y,\bbZ)$ for $i>2$.  Since $Y$ is a Stein space, the latter cohomology groups vanish (see \cite{Andreotti-Narasimhan}, p.~500 and 508-509 for more details). In particular $H_3(\wt{M},\bbZ)=0$ as desired. It follows that the map $f:M\to N$ is a homotopy equivalence, and consequently it follows easily that $f$ is biholomorphic; see \cite[Thm.~8]{Siu}.
\end{pf}

\begin{rem}\label{rem:pinf2}
The hypothesis of Theorem~\ref{thm:pi2=0} can be weakened to only assume that $\pi_2(M)$ is finite, since Gurjar proved that $\pi_2(M)$ is torsion-free when $\wt{M}$ is holomorphically convex \cite[Thm.~1]{Gurjar}.
\end{rem}

\begin{rem}\label{rem:delzant}
The proof of Theorem \ref{thm:pi2=0} suggests that there should be wide class of algebraic surfaces $X$ for which $\pi_2(X) = 0$ implies that $X$ is aspherical. A careful look at the proof shows that this is the case under the hypothesis that the universal cover $\wt{X}$ is holomorphically convex and its Cartan--Remmert reduction is two-dimensional. It then looks like the Shafarevich conjecture comes into play, and perhaps finding non-aspherical surfaces $X$ with $\pi_2(X) = 0$ would produce a counterexample to the Shafarevich conjecture. However this is not the case. Pierre Py informed us of the following unpublished result of Thomas Delzant, and explained the simple and elegant proof.
\end{rem}

\begin{thm}[Delzant]\label{thm:delzant}
Let $X$ be a compact K\"ahler surface and assume that $\pi_1(X)$ is not commensurable to a surface group. Then $X$ is aspherical if and only if $\pi_2(X) = 0$.
\end{thm}

\begin{proof}[Sketch of Proof]
Let $\wt{X}$ be the universal cover of $X$. If
\[
\pi_2(X) = H_2(\wt{X},\bbZ) = 0,
\]
then, applying the Hurewicz theorem, $X$ is aspherical if and only if $H_3(\wt{X,}\bbZ)$ is zero. Since $\wt{X}$ is homotopically a $3$-complex, this is equivalent to triviality of $H_3(\wt{X},\bbR)$, so suppose $H_3(\wt{X},\bbR) \ne 0$. By duality, the compactly supported cohomology $H^1_c(\wt{X},\bbR)$ is nontrivial, which in turn, implies that the one-dimensional $L^2$ cohomology $H^1_{(2)}(\wt{X})$ is nontrivial. By a celebrated theorem of Gromov \cite{GromovL2}, this implies that $\pi_1(X)$ is commensurable with a surface group.
\end{proof}

\section{Realizing central extensions}\label{sec:Realize}

Fix $d \ge 2$ and suppose that $N = \Gam \bs \bbB^2$ is a smooth compact ball quotient. Our goal is to construct a smooth projective surface $M$ with fundamental group that fits into a central exact sequence:
\begin{equation}\label{eq:dCentral}
1 \lra \bbZ / d \lra \pi_1(M) \lra \Gam \lra 1
\end{equation}
The known constructions all appear to be roughly equivalent to an ``extraction of roots'', where $M$ is a branched cover of $N$. We now explain one variant  of \lq\lq extraction of roots", sketched in \cite[Ex.~8.15]{ABCKT}. We begin by describing some auxiliary objects and establishing notation.

\subsection{The projective space of a direct sum}\label{ssec:Psum}

If $V$ is a complex vector space and $v\in V, v\ne 0$, $[v]\in \bbP(V)$ denotes the line determined by $v$.  If $V_1, V_2$ are complex vector spaces, $v_1\in V_1$, $v_2\in V_2$, $v_1$ and $v_2$ not both $0$, we write simply $[v_1,v_2]$  for the element $[(v_1,v_2)]\in\bbP(V_1\oplus V_2)$. The case of interest for us will be $V_1 \cong V_2 \cong \bbC^3$.

Consider $\bbP(V_1)$ and $\bbP(V_2)$ as subsets of $\bbP(V_1\oplus V_2)$ in the standard way
\begin{align*}
[v_1] &\mapsto [v_1, 0] & [v_2] &\mapsto [0,v_2]
\end{align*}
and the general element of $\bbP(V_1\oplus V_2)$ is therefore of the form $[v_1,v_2]$ with $v_1\in V_1, v_2\in V_2$, not both $0$, thus presenting $\bbP(V_1 \oplus V_2)$ as a \lq\lq join" of $\bbP(V_1)$ and $\bbP(V_2)$.

If we set
\[
U = \{[v_1, v_2] \in \bbP(V_1 \oplus V_2)~:~v_1\neq 0\ \text{and}\  v_2 \neq 0\},
\]
then we have a decomposition
\[
\bbP(V_1 \oplus V_2) = U \sqcup \bbP(V_1) \sqcup  \bbP(V_2)
\]
that we can visualize as in Figure~\ref{fig:UP1P2}.
\begin{figure}[h]
\centering
\begin{tikzpicture}
\draw[very thick, blue] (-2,2) -- (2,-2);
\draw[white, fill=white] (0,0) circle (0.25cm);
\draw[very thick, blue] (-2,-1) -- (4,2);
\draw[very thick, orange] (2,1) -- (5/4,-5/4);
\node[blue, right] at (4,2) {$\bbP(V_1)$};
\node[blue, right] at (2,-2) {$\bbP(V_2)$};
\draw[orange, fill=orange] (13/8,-1/8) circle (0.05cm) node[right] {$[v_1, v_2]$} node[below left] {$U$};
\node[blue, above left] at (2,1) {$[v_1]$};
\draw[blue, fill=white] (2,1) circle (0.07cm);
\node[blue, below left] at (5/4,-5/4) {$[v_2]$};
\draw[blue, fill=white] (5/4,-5/4) circle (0.07cm);
\end{tikzpicture}
\caption{$\bbP(V_1 \oplus V_2) = U \sqcup \bbP(V_1) \sqcup \bbP(V_2)$}\label{fig:UP1P2}
\end{figure}

This determines a fibration
\begin{equation}\label{eq:Ufibration}
\begin{tikzcd}[column sep = small]
\bbC^* \arrow{r} & U \arrow{d} & {[v_1, v_2]} \arrow[mapsto]{d} \\
 & \bbP(V_1) \times \bbP(V_2) & {\left([v_1], [v_2] \right)}
\end{tikzcd}
\end{equation}
where the fiber over $([v_1], [v_2])$, which we denote by $\mathopen< [v_1],[v_2] \mathclose>$, is
\begin{equation}\label{eq:Ufiber}
\mathopen< [v_1],[v_2] \mathclose> = \{[\lam v_1, \mu v_2]~:~\lam, \mu \in \bbC^*\} 
\end{equation}
in other words, the projective line joining the points $[v_1,0], [0,v_2]$, with these two points removed. This is isomorphic to $ \left(\bbC^* \times \bbC^*\right) / \textrm{diagonal}$ which is in turn is isomorphic to $ \bbC^*$, but in more than one way, since $0$ and $\infty$ can be interchanged.

To describe this fibration in terms of familiar line bundles, let
\[
\begin{tikzcd}
L_j \arrow[hookrightarrow]{r} \arrow{d} & \bbP(V_j) \times V_j \\
\bbP(V_j)
\end{tikzcd}
\]
be the tautological line bundles, $j = 1,2$. Now, let $\bbL_j$ be the pullback of $L_j$ to $\bbP(V_1) \times \bbP(V_2)$ under the projection to $\bbP(V_j)$. Letting $L^{-1}$ denote the inverse of a line bundle and $L^*$ the line bundle minus its zero section, we have identifications
\begin{equation}\label{eq:Utaut}
U = \left( \bbL_1 \otimes \bbL_2^{-1} \right)^* = \left( \bbL_1^{-1} \otimes \bbL_2 \right)^*,
\end{equation}
which one can see directly in coordinates from the calculation
\[
\Big[\lam v_1, \mu v_2 \Big] = \Big[v_1, \frac{\mu}{\lam} v_2 \Big] = \Big[\frac{\lam}{\mu} v_1, v_2 \Big]
\]
for $[v_j] \in \bbP(V_j)$ and $\lam, \mu \neq 0$. Specifically, the first equality identifies the point $[\lambda v_1, \mu v_2]$  over $([v_1],[v_2])$ with the graph of the linear map
\[
\left(v_1 \mapsto \frac{\mu}{\lambda} v_2\right) \in \Hom(\bbL_1,\bbL_2) \cong \bbL_1^{-1}\otimes \bbL_2,
\]
and similarly with the second identification.

\subsection{Finite maps to $\bbP^N$}\label{ssec:FiniteMap}

This section refines the well-known connection between very ample line bundles and maps to projective spaces. Critical for our purposes is that one can produce a \emph{finite} map, and we include the proof of this fact for completeness.

\begin{lem}\label{lem:FiniteMap}
Suppose that $X$ is an $n$-dimensional smooth projective variety and $\eta \in H^2(X, \bbZ)$ is the first Chern class of a very ample line bundle on $X$. Then there exists a finite map
\[
\al : X \lra \bbP^n
\]
so that $\al^*([H]) = \eta$, where $[H] \in H^2(\bbP^n, \bbZ)$ is the Poincar\'e dual to a hyperplane.
\end{lem}

\begin{pf}
Let $\al_0 : X \to \bbP^N$, $N > n$, be an embedding of $X$ in $\bbP^N$ determined by a basis for the space of global sections of the very ample line bundle. Then $\al_0^*([H_0]) = \eta$, where $[H_0] \in H^2(\bbP^N, \bbZ)$ is the hyperplane class. See \cite[\S 1.4]{GriffithsHarris} for a standard reference.

Now, choose $p_0 \in \bbP^N$ so that $p_0 \notin \al_0(X)$ and the generic hyperplane $H_0 \subset \bbP^N$ though $p_0$ is transverse to $\al_0(X)$. Then $\al_0^{-1}(H_0)$ is Poincar\'e dual to $\al_0^*([H_0]) = \eta$. Let
\[
\pi_0 : \bbP^N \ssm \{p_0\} \lra \bbP^{N-1}
\]
be projection centered at $p_0$ and $\al_1 = \pi_0 \circ \al_0$. Then, by construction, $\al_1$ is a finite map. Since $\pi_0^{-1}$ determines a correspondence between hyperplanes in $\bbP^{N-1}$ and hyperplanes in $\bbP^N$ through $p_0$, we have that
\[
\al_1^{-1}(H_1) = \al_0^{-1}(H_0)
\]
for all hyperplanes $H_1 \subset \bbP^{N-1}$, where $H_0 = \pi_0^{-1}(H_1)$. Thus, for a general $H_1 \subset \bbP^{N-1}$, we see that $\al_1^{-1}(H_1)$ is Poincar\'e dual to $\eta$, i.e., $\al_1^*([H_1]) = \eta$.

We now iterate this procedure to obtain $\al_j : X \to \bbP^{N-j}$. As long as $N-j > n$, we can find a point $p_j \notin \al_j(X)$ that allows us to construct $\al_{j+1}$. Note that the composition of finite maps is finite, so each $\al_j$ is a finite map. This ends with $\al = \al_{N-n} : X \to \bbP^n$, which is the desired map.
\end{pf}

\subsection{The construction of $M$}\label{ssec:Construct}

We now build the smooth projective surface $M$ described at the beginning of this section. Let $N = \Gam \bs \bbB^2$ be a smooth compact ball quotient. Fix a very ample line bundle on $N$, e.g., an appropriate power of the canonical bundle, and let
\[
\al : N \to \bbP^2
\]
be the map provided by Lemma~\ref{lem:FiniteMap}. If $H$ denotes a generic hyperplane on $\bbP^2$, then $\al^*([H])$ is the first Chern class of our chosen very ample line bundle. Finally, let
\[
\begin{tikzcd}[column sep = small]
\bbC^* \arrow{r} & Z \arrow{d} \\ & N
\end{tikzcd}
\]
denote the $\bbC^*$ bundle over $N$ obtained by removing the zero section from the total space of our chosen very ample line bundle, and set $\wt{\Gam} = \pi_1(Z)$. Note that the element of $H^2(\Gam, \bbZ) \cong H^2(N, \bbZ)$ associated with realizing $\wt{\Gam}$ as a central extension of $\Gam$ by $\bbZ$ is the same as the first Chern class of our line bundle.

\medskip

Fix $d \ge 2$, and let $\beta : \bbP^2 \to \bbP^2$ be the $d^{th}$ power map
\begin{equation}
\label{eq:betadef}
[x_0 : x_1 : x_2] \mapsto [x_0^d : x_1^d : x_2^d].
\end{equation}
Note that $\beta^*([H]) = d[H]$, and that $\beta$ is a $(\bbZ / d)^2$ branched cover. From \S\ref{ssec:Psum}, we have the fibration
\begin{equation}\label{eq:Vbundle}
\begin{tikzcd}[column sep = small]
V \arrow[dashed]{d} \arrow[dashed]{rr}{F} & & U \arrow{d} \arrow[hookrightarrow]{r} & \bbP^5 \\
N \times \bbP^2 \arrow{rr}{\al \times \beta} & & \bbP^2 \times \bbP^2 &
\end{tikzcd}
\end{equation}
where $V$ is the pullback to $N \times \bbP^2$ of the $\bbC^*$ bundle $U$ over $\bbP^2 \times \bbP^2$. More precisely,
\begin{equation}
\label{eq:Upullback}
V = \left\{(x, y, u)~:~x \in N, y \in \bbP^2, u \in \mathopen<\al(x) , \beta(y) \mathclose>\right\},
\end{equation}
where $\mathopen<\al(x) , \beta(y) \mathclose>$ is as defined in Equation~\eqref{eq:Ufiber} and 
\begin{equation}
\label{eq:Fmap}
F(x,y,u) = u
\end{equation}
for $(x,y,u)\in V$.

The following theorem is a very special case of a theorem of Goresky and McPherson \cite[Thm.\ II.1.1]{Goresky-MacPherson}. 
\begin{thm}\label{thm:Morse}
Let $V$ be a smooth algebraic variety and $f : V \to \bbP^N$ be an algebraic map with finite fibers. If $\Lambda \subset \bbP^N$ is a generic linear subspace, then $\pi_i(V, f^{-1}(\Lambda)) = 0$ for $i \le \dim(\Lambda)$.
\end{thm}

We will apply Theorem~\ref{thm:Morse} to our $F : V \to U \subset \bbP^5$ defined in Equation~\eqref{eq:Vbundle} and $\Lambda\subset U$ a generic $\bbP^2$. First, observe that any $2$-dimensional projective subspace
\[
\Lambda \subset U\subset \bbP^5
\]
is disjoint from $\bbP(V_1)$ and $\bbP(V_2)$. Consequently, it is the graph of a linear isomorphism $A:V_1\to V_2$, that is,
\[
\Lambda = \left\{[v_1, A v_1]~:~v_1 \in V_1\right\}.
\]
We now describe $F^{-1}(\Lambda)$ more precisely.

To find $F^{-1}(\Lambda)$, note that for any $(x,y)\in N\times \bbP^2$, $\mathopen<\alpha(x),\beta(y)\mathclose> \cong \bbC^*$ can meet $\Lambda$ in at most one point. Indeed, otherwise $\mathopen<\alpha(x),\beta(y) \mathclose>\subset \Lambda$, and $\Lambda$ would then contain the points in each of $\bbP(V_1)$ and $\bbP(V_2)$ contained in the closure of $\mathopen<\alpha(x),\beta(y) \mathclose>$, contradicting the fact that $\Lambda\subset U$.

Moreover, $\mathopen<\alpha(x),\beta(y) \mathopen>\,\cap\, \Lambda \neq \varnothing$ if and only if there exist $\widetilde{\alpha(x)}\in V_1$ and $\widetilde{\beta(y)}\in V_2$ in the lines $\alpha(x)\in \bbP(V_1)$ and $\beta(y)\in\bbP(V_2)$, respectively, such that $\widetilde{\beta(y)} = A \widetilde{\alpha(x)}$. This is the same as saying that $\beta(y) = A \alpha(x)$, where $A:\bbP(V_1)\mapsto \bbP(V_2)$ is the map of projective spaces induced by $A$. This is moreover equivalent to saying $\alpha(x) = A^{-1} \beta(y)$. In other words, if we define
\begin{align}\label{eq:Mdef}
M &= \left\{(x, y) \in N \times \bbP^2~:~\beta(y) = A \al(x) \right\} \\
&= \left\{(x,y) \in N \times \bbP^2~:~\al(x) = A^{-1} \beta(y) \right\} \nonumber
\end{align}
we have seen that $\mathopen<\alpha(x),\beta(y) \mathclose> \cap \Lambda \neq\varnothing$ if and only if $(x,y)\in M$. Conversely, if $(x,y)\in M$ then there exits a unique point $u(x,y)$ contained in $\mathopen<\alpha(x),\beta(y)\mathclose> \cap \Lambda$.

Therefore,
\[
F^{-1}(\Lambda) = \{ (x,y,u(x,y))\in V ~:~ (x,y)\in M\} = \wh{M}
\]
and $F^{-1}(\Lambda)\subset V$. From now on we denote $F^{-1}(\Lam)$ by $\wh{M}$, which is the image of a section over $M$ of the $\bbC^*$-bundle $V \to N\times \bbP^2$ of Equation~\eqref{eq:Vbundle}. In particular, $\wh{M}$ is biholomorphic to $M$. In terms of the diagram in Equation~\eqref{eq:Vbundle}, $M$ and $\wh{M}$ fit as
\begin{equation}\label{eq:Finverse}
\begin{tikzcd}[column sep = small]
\wh{M}=F^{-1}(\Lambda) \arrow{d} \arrow[hook]{r} & V \arrow[dashed]{d} \arrow{rr}{F} & & U \arrow{d} \arrow[hookrightarrow]{r} & \bbP^5 \\
M \arrow[hook]{r} & N \times \bbP^2 \arrow{rr}{\al \times \beta} & & \bbP^2 \times \bbP^2 &
\end{tikzcd}
\end{equation}
where the left-most vertical arrow is a biholomorphism. Note that $\Lambda$, $M$, and $\wh{M}$ depend on the linear isomorphism $A:V_1\to V_2$.

\begin{lem}\label{lem:Mfiber}
With notation as above, given a generic linear isomorphism $A:V_1\to V_2$, the complex surfaces $\wh{M} = F^{-1}(\Lambda)\subset V$ and its projection $M\subset N\times \bbP^2$ defined in Equation~\eqref{eq:Mdef} are biholomorphic smooth projective varieties. Further, $M$ can be described as a fiber product by either of the diagrams
\[
\begin{tikzcd}
M \arrow{r} \arrow{d} & \bbP(V_2) \arrow{d}{A^{-1} \beta} & & M \arrow{r} \arrow{d} & \bbP(V_2) \arrow{d}{\beta} \\
N \arrow{r}{\al} & \bbP(V_1) & & N \arrow{r}{A \al} & \bbP(V_2)
\end{tikzcd}
\]
in which the vertical maps are $(\bbZ/d)^2$-branched covers.
\end{lem}

\begin{pf}
We already proved that $\wh{M} = F^{-1}(\Lambda)\subset V$ projects to $M\subset N\times \bbP^2$ as in the diagram from Equation~\eqref{eq:Finverse} and that $\wh{M}$ is the image of a section over $M$ of the bundle $V\to N\times \bbP^2$, hence the projection is biholomorphic. By its definition from Equation~\eqref{eq:Mdef}, $M$ is a closed subvariety of the projective variety $N\times \bbP^2$, hence it is projective and the squares are fiber-squares. Since $\beta$ is a branched $(\bbZ/d)^2$-cover by Equation~\eqref{eq:betadef}, so are the other vertical maps.

So far, the above is true for any $A$. It remains to prove that, for generic $A$, $M$ is a smooth subvariety of $N\times \bbP^2$. To see this, the map
\begin{align*}
\GL_3(\bbC) \times N \times \bbP^2 &\lra N \times \bbP^2 \times \bbP^2 \\
(A, x, y) &\overset{f}{\longmapsto} (x, A \al(x), y)
\end{align*}
is a submersion. Therefore, Sard's theorem implies that the map $f_A(x,y) = (x, A \al(x), y)$ is transverse to the submanifold
\[
Y = \left\{(x,\beta(y),y)~:~x\in N, y \in \bbP^2 \right\} \subset N \times \bbP^2 \times \bbP^2
\]
for almost every $A \in \GL_3(\bbC)$. Then $f_A(N \times \bbP^2) \cap Y$ is precisely the above fiber product associated with the given choice of $A$, hence we see that the fiber product is a manifold for almost every choice of $A$. Said otherwise, $M$ is a manifold for almost every $\Lambda$. This completes the proof.
\end{pf}

We now show that $\pi_1(M)$ fits into the desired central exact sequence involving $\Gam$. We will also compute $\pi_2(M)$ under the hypothesis that $\pi_1(M)$ is residually finite, which we will use to show that $M$ is not a modification of a ball quotient, i.e., that it cannot be obtained from a ball quotient by a sequence of blowups and blowdowns.

\begin{thm}\label{thm:DescribeM}
With notation as above, the smooth projective surface $M$ has the following properties.
\begin{enum}
\item The fundamental group of $M$ fits into a central exact sequence:
\begin{equation}\label{eq:pi1M}
1 \lra \bbZ / d \lra \pi_1(M) \lra \Gam \lra 1
\end{equation}

\item The exact sequence in Equation~\eqref{eq:pi1M} is the reduction modulo $d$ of the central exact sequence
\begin{equation}\label{eq:pi1N}
1 \lra \bbZ \lra \wt{\Gam} \lra \Gam \lra 1
\end{equation}
for the fundamental group $\wt{\Gam}$ of $Z$.

\item If $\pi_1(M)$ is residually finite, then $\pi_2(M) \neq \{0\}$.

\end{enum}
\end{thm}

\begin{pf}[Proof of parts 1 and 2 of Theorem~\ref{thm:DescribeM}]
Since $\alpha$ and $\beta$ are finite maps and $F:V\to U$ is a bundle map (that is an isomorphism on the fibers of the fibrations in Equation~\eqref{eq:Vbundle}), it follows that the map $F:V\to U\subset\bbP^5$ has finite fibers.  Thus Theorem~\ref{thm:Morse} applies and gives us $\pi_i(V,\wh{ M}) = 0$ for $i \le 2$. The homotopy exact sequence
\[
\begin{tikzcd}
\pi_2( \wh{M}) \arrow[r] & \pi_2(V) \arrow[r] \arrow[d,phantom, ""{coordinate, name=Z}] & \pi_2(V, \wh{M}) \cong \{0\} \arrow[dll,rounded corners,to path={ --([xshift=2ex]\tikztostart.east)|- (Z)[near end]\tikztonodes-| ([xshift=-2ex]\tikztotarget.west)-- (\tikztotarget)}] \\
\pi_1 (\wh{M}) \arrow[r] & \pi_1(V) \arrow[d,phantom, ""{coordinate, name=Y}] \arrow[r] & \pi_1(V,  \wh{M}) \cong \{0\} \arrow[dll,rounded corners,to path={ --([xshift=2ex]\tikztostart.east)|- (Y)[near end]\tikztonodes-| ([xshift=-2ex]\tikztotarget.west)-- (\tikztotarget)}] \\
\pi_0( \wh{M}) \arrow[r] & \pi_0(V) \cong \{0\}
\end{tikzcd}
\]
then implies that $ \wh{M}$ is connected, $\pi_1( \wh{M}) \cong \pi_1(V)$, and $\pi_2( \wh{M})$ maps onto $\pi_2(V)$.

In the last three statements, we can replace $\wh{M}$ by $M$, since they are biholomorphic by Lemma~\ref{lem:Mfiber}. We therefore want to compute the homotopy groups of $V$. Considering $V$ as a $\bbC^*$ fibration over $N \times \bbP^2$, we have a the exact sequence:
\[
\begin{tikzcd}
\pi_2(\bbC^*) \cong \{0\} \arrow[r] & \pi_2(V) \arrow[r] \arrow[d,phantom, ""{coordinate, name=Z}] & \pi_2(N \times \bbP^2) \cong \pi_2(\bbP^2) \cong \bbZ \arrow[dll,rounded corners,to path={ --([xshift=2ex]\tikztostart.east)|- (Z)[near end]\tikztonodes-| ([xshift=-2ex]\tikztotarget.west)-- (\tikztotarget)}] & \\
\pi_1(\bbC^*) \cong \bbZ \arrow[r] & \pi_1(V) \arrow[r] & \pi_1(N \times \bbP^2) \cong \pi_1(N) \arrow[r] & \{0\}
\end{tikzcd}
\]
Once we understand the map $\pi_2(\bbP^2) \to \pi_1(\bbC^*)$, computing $\pi_i(V)$ for $i \le 2$ will be direct.

For a fixed $y \in N$ and $z \in \bbP^2$, $F$ restricts to fibrations isomorphic to:
\[
\begin{tikzcd}[column sep = small]
\bbC^* \arrow{r} & \left(\al^* \bbL_1^{-1} \right)^* \arrow{d} & & \bbC^* \arrow{r} & \left(\beta^* \bbL_2^{-1} \right)^* \arrow{d} \\
& N \times \{z\} & & & \{y\} \times \bbP^2
\end{tikzcd}
\]
To simplify notation, let $\calL_1$ and $\calL_2$ denote the total spaces of these $\bbC^*$ bundles over $N \times \{z\}$ and $\{y\} \times \bbP^2$, respectively. Since $\beta^*\bbL_2$ is the line bundle $\calO_{\bbP^2}(-d)$ by construction, $\calL_2$ is $\calO_{\bbP^2}(d)$ minus its zero section. Since $\pi_2(\bbC^*) = \{0\}$, we have homomorphisms of exact sequences
\begin{equation}\label{eq:BigHomotopy}
{\small
\begin{tikzcd}[column sep = small]
\{0\} \arrow{r} & \pi_2(\calL_1) \arrow{r} \arrow{d} & \overset{\textcolor{magenta}{\raisebox{1.5 ex}{$\{0\}$}}}{\textcolor{orange}{\pi_2(N)}} \arrow[orange]{r} \arrow[orange]{d} & \overset{\textcolor{magenta}{\raisebox{1.5 ex}{$\bbZ$}}}{\textcolor{orange}{\pi_1(\bbC^*)}} \arrow[orange]{r} \arrow[orange]{d}{\cong} & \textcolor{orange}{\pi_1(\calL_1)} \arrow[orange]{r} \arrow[orange]{d} & \overset{\textcolor{magenta}{\raisebox{1.5 ex}{$\Gam$}}}{\textcolor{orange}{\pi_1(N)}} \arrow[orange]{r} \arrow[orange]{d}{\cong} & \textcolor{orange}{\{0\}} \\
\{0\} \arrow{r} & \pi_2(V) \arrow{r} & \textcolor{blue}{\pi_2(N \times \bbP^2)} \arrow[blue]{r} & \textcolor{blue}{\pi_1(\bbC^*)} \arrow[blue]{r} & \textcolor{blue}{\pi_1(V)} \arrow[blue]{r} & \textcolor{blue}{\pi_1(N \times \bbP^2)} \arrow[blue]{r} & \textcolor{blue}{\{0\}} \\
\{0\} \arrow{r} & \underset{\textcolor{magenta}{\raisebox{-2 ex}{$\{0\}$}}}{\pi_2(\calL_2)} \arrow{r} \arrow{u} & \underset{\textcolor{magenta}{\raisebox{-2 ex}{$\bbZ$}}}{\textcolor{blue}{\pi_2(\bbP^2)}} \arrow[blue]{r}[below, blue]{\times d} \arrow[blue]{u}{\cong} & \underset{\textcolor{magenta}{\raisebox{-2 ex}{$\bbZ$}}}{\textcolor{blue}{\pi_1(\bbC^*)}} \arrow{r} \arrow[blue]{u}{\cong} & \underset{\textcolor{magenta}{\raisebox{-2 ex}{$\bbZ / d$}}}{\pi_1(\calL_2)} \arrow{r} \arrow{u} & \underset{\textcolor{magenta}{\raisebox{-2 ex}{$\{0\}$}}}{\pi_1(\bbP^2)} \arrow{r} \arrow{u} & \{0\}
\end{tikzcd}
}
\end{equation}
where the bottom row is the standard homotopy exact sequence for the complement of the zero section in $\calO_{\bbP^2}(d)$.

Following the \textcolor{blue}{blue} portion of Equation~\eqref{eq:BigHomotopy}, we see that $\pi_1(V)$ fits into an exact sequence:
\begin{equation}\label{eq:pi1V}
1 \lra \bbZ / d \lra \pi_1(V) \lra \Gam \lra 1
\end{equation}
From the \textcolor{orange}{orange} portion, we see that Equation~\eqref{eq:pi1V} is the reduction modulo $d$ of the exact sequence
\begin{equation}\label{eq:pi1L1}
1 \lra \bbZ \lra \pi_1(\calL_1) \lra \Gam \lra 1
\end{equation}
for the fundamental group of the total space of the line bundle over $N$ with first Chern class $\eta$ minus its zero section. Since $\pi_1(V) \cong \pi_1(M)$, this proves parts 1 and 2 of the theorem.
\end{pf}

\begin{rem}
The above argument shows that $\pi_2(M)$ surjects $\pi_2(V)$, but that $\pi_2(V)$ is trivial. Thus we cannot conclude directly from $V$ that $\pi_2(M)$ is nontrivial.
\end{rem}

We can now simultaneously prove part 3 of Theorem~\ref{thm:DescribeM} and Theorem~\ref{thm:dFold}. First, the $M$ in part 1 of Theorem~\ref{thm:DescribeM} gives us the $M(\Gam,d)$ of Theorem~\ref{thm:dFold}, thus proving the first part of Theorem~\ref{thm:dFold} except for the statement about $\pi_2$. Next, if  $\Gam_d = \pi_1(M)$ is residually finite, Lemma~\ref{lem:CentralRF} gives us the subgroup $\Lam\subset \Gam$ and the section $s:\Lam\to \Gam_d$, thus proving the second part of Theorem~\ref{thm:dFold}.  
 
To prove the remaining statements in both theorems, note that if $\Gam_d = \pi_1(M)$ is residually finite, then there are finite \'etale covers $M^\prime \to M$ and $N^\prime \to N$ such that the finite map $f : M \to N$ defined by the fibre squares in the statement of Lemma~\ref{lem:Mfiber} lifts to a map $\wh{f} : M^\prime \to N^\prime$ that induces an isomorphism on fundamental groups. These are the covers corresponding to the subgroups $s(\Lam) \subset \Gam_d = \pi_1(M)$ and $\Lam\subset\Gam =\pi_1(N)$ respectively.

Note that $\pi_2(M^\prime) \cong \pi_2(M)$. If $\pi_2(M)$ was zero, then Theorem~\ref{thm:pi2=0} implies that $M^\prime$ is a ball quotient, so in particular it is aspherical. The map $\wh{f}:M^\prime \to N$ inducing an isomorphism of fundamental groups would then be a homotopy equivalence, in particular have degree $\pm 1$. However, if $M^{\prime\prime}$ is the cover of $M$ associated with the subgroup $\pi_d^{-1}(\Lam)\subset\Gam_d$, isomorphic to $\bbZ / d \times s(\Lam)$ containing $s(\Lam)$ (with $\pi_d : G_d \to G$ the covering), we obtain a diagram of maps
\[
\begin{tikzcd}
M \arrow{r}{\deg\, d^2} & N \\
M^{\prime\prime} \arrow{r} \arrow{u} & N^\prime \arrow{u} \\
M^\prime \arrow{u}[left]{d\textrm{-to-}1} \arrow[dashed]{ur}[below]{\wh{f}} &
\end{tikzcd}
\]
This shows that $\wh{f}:M^\prime\to N$ has degree $d^3 >1$, contradicting the previous calculation of this degree. This contradiction implies that $\pi_2(M)$ must be nontrivial (in fact infinite by Remark~\ref{rem:pinf2}), completing the proof of both theorems except for the statement that $M^\prime$ is not birational to any ball quotient.

Suppose $Z$ is a compact, smooth quotient.of $\bbB^2$.  Observe that $Z$ has no rational curves, that is, every holomorphic map $\bbP^1 \to Z$ is constant.  Note  that the same is true for $M^\prime$, because, by construction,  $M^\prime$ has a finite map to a ball quotient $N$. We use the following fact (see \cite[Thm.~VI.1.9.3]{KollarRat}):

\begin{lem}\label{lem:KollarRat}
Suppose $X,Y$ are smooth projective varieties, suppose $Y$ has no rational curves, and $f:X\dashrightarrow Y$ is a rational map. Then $f$ is everywhere defined.
\end{lem}

Finally, suppose $f:M^\prime \dashrightarrow Z$ is a birational map.  Applying Lemma~\ref{lem:KollarRat} to $f$ and $f^{-1}$ we see that they are both everywhere defined, so $f$ is biholomorphic, which is impossible. Thus $M^\prime$ and $Z$ cannot be birational. This completes the proofs of part 3 of Theorem~\ref{thm:DescribeM} and Theorem~\ref{thm:dFold}.

\medskip

To give a concrete example of how one can use the methods of this section, we now prove Theorem~\ref{thm:FakeFakes}.

\begin{pf}[Proof of Theorem~\ref{thm:FakeFakes}]
Let $\Gam$ be the fundamental group of a fake projective plane $X$ whose canonical bundle $K_X$ is divisible by three. The case $d = 3$ of Theorem~\ref{thm:DescribeM} applied to the preimage $\wt{\Gam}$ of $\Gam$ in the universal cover of $\PU(2,1)$ provides a smooth complex surface $S$ whose fundamental group is the central $\bbZ / 3$ extension of $\Gam$ isomorphic to the preimage of $\Gam$ in $\SU(2,1)$. Our assumption on $K_X$ implies that this lift splits, so $\pi_1(S) \cong \bbZ / 3 \times \Gam$ (e.g., see \cite[\S 8.9]{Kollar}).

It follows that there is a degree $3$ \'etale cover $S^\prime \to S$ with fundamental group $\Gam$. By construction, $S^\prime$ admits a degree $9$ map onto $X$, hence there is a degree $27$ map from $S^\prime$ to $X$. Arguing  as in the proof of part 3 of Theorem~\ref{thm:DescribeM}, this implies that $S^\prime$ cannot be a modification of $X$. Thus $S^\prime$ is the desired surface.
\end{pf}

\begin{rem}\label{rem:FakeChern}
We briefly explain how one can determine the Chern numbers of the surfaces in Theorem~\ref{thm:FakeFakes}. Let $X$ be a fake projective plane such that $K_X$ is divisible by $3$. We need to replace $K_X$ with a very ample line bundle with the same first Chern class as $K_X$ modulo $3$. Using Reider's theorem as in \cite[Thm.\ 2.1]{CataneseKeum}, one easily sees that $n K_X$ is very ample for all $n \ge 3$, and it has the same Chern class modulo $3$ as $K_X$, since $K_X$ is divisible by $3$. In fact, Catanese and Keum prove that one can take $2 K_X$ for many fake projective planes \cite[Thm.\ 1.1]{CataneseKeum}.

Proceeding with our construction, we take $S_n$ to be the surface with fundamental group $\pi_1(X) \times \bbZ / 3$ built above. The construction of $S_n$ as a fiber product implies that it is the $(\bbZ / 3)^2$ branched cover of $X$ over three general members $D_1, D_2, D_3$ of the linear system $|n K_X|$. Note that ``general'' here is only needed to ensure that the triple intersection is empty. We compute the genus of $D_j$ by adjunction:
\begin{align*}
2 (g - 1) &= K_X D_j + D_j^2 \\
&= n K_X^2 + n^2 K_X^2 \\
&= 9 n(n+1)
\end{align*}
Each $D_j$ meets another of the curves $D_k$ in $D_j D_k = 9 n^2$ points, hence each $D_j$ contains $18 n^2$ singular points and there are $27 n^2$ singular points in total.

Branching of $S_n \to X$ has order $3$ over each nonsingular point of $D_j$ and order $9$ over the singular points. For example, following \cite[p.\ 16-17]{BHH}, we can then compute the Chern numbers of $S_n$ by:
\begin{align*}
c_1^2(S_n) &= 9\left(K_X + 3 \left(1-\frac{1}{3} \right) n K_X \right)^2 \\
&= 27\, \left((2n+1) K_X\right)^2 \\
&= 243\, (2n+1)^2 \\
c_2(S_n) &= 9\left(3 - 3 \left(1 - \frac{1}{3}\right)(-9n(n+1)-18 n^2) - 27 n^2 \left(1 - \frac{1}{9}\right)\right) \\
&= 81\, (10 n^2 + 6 n + 1)
\end{align*}
Now, let $S_n^\prime$ be the \'etale $\bbZ / 3$ cover of $S_n$ with $\pi_1(S_n^\prime) \cong \pi_1(X) < \pi_1(S_n)$. Note that the Chern numbers of $S_n^\prime$ are three times those of $S_n$. This determines a sequence of surfaces $S_n^\prime$, $n \ge 1$, where $S_1^\prime = X$, maybe $S_2^\prime$ is not defined if $2 K_X$ is not very ample, and for $n \neq 1$ we have:
\[
c_1^2(S_n^\prime) / c_2(S_n^\prime) = \frac{3(2n+1)^2}{10 n^2 + 6n + 1} \overset{n\to\infty}{\xrightarrow{\hspace*{0.7cm}}} \frac{6}{5}
\]
Note that Troncoso and Urz\'ua \cite{TroncosoUrzua} produce examples of surfaces $Y_\al$ with $\pi_1(Y_\al) \cong \pi_1(X)$ and $c_1^2(Y_\al) / c_2(Y_\al)$ dense in $[1,3]$, whereas we only obtain the limit point $1.2$.
\end{rem}

\section{Examples}\label{sec:Ex}

This section provides the examples necessary to complete the proof of Theorem~\ref{thm:ExsExist} using presentations of preimages in the universal cover. In particular, we use the method developed in \S\ref{sec:Lift}. Supplemental Mathematica files are available on the first author's website for readers interested in verifying some of the calculations that follow. Our examples are all commensurable with the lattices constructed by Deligne and Mostow \cite{DeligneMostow1, MostowSigma}. Throughout this section, $\zeta_n$ will denote a primitive $n^{th}$ root of unity.

We study Deligne--Mostow lattices $\Gam < \PU(2,1)$ for which the underlying space for the orbifold $\Gam \bs \bbB^2$ is the weighted projective space $\bbP(1,2,3)$, i.e., the quotient of $\bbP^2$ by the natural action of $S_3$ by coordinate permutations. A basic reference for the underlying spaces for Deligne--Mostow orbifolds is \cite[Thm.\ 4.1]{KirwanLeeWeintraub}; in their notation, we are interested in the cases where $Q(\mu)$ is $\bbP^2$ and $\mu$ has symmetry group $S_3$.

We begin with some basics on a standard presentation for these lattices based on the topology of $\bbP(1,2,3)$. If $[x_1 : x_2 : x_3]$ denotes homogeneous coordinates on $\bbP^2$, let $A$ be the image on $\bbP(1,2,3)$ of the lines $x_i = x_j$ and $B$ denote the images of the coordinate lines $x_i = 0$. Then $A$ and $B$ are tangent at one point, intersect transversally at another, $B$ has a cusp singularity, there is a singularity of $\bbP(1,2,3)$ of type $\mathrm{A}_1$ on $A$, and a singularity of type $\mathrm{A}_2$ that lies on neither $A$ or $B$. See Figure~\ref{fig:123} for a visualization.
\begin{figure} [htbp]
\centering
\begin{tikzpicture}(-2.5,0.95) (-0.75,0.725)
\draw[white, fill=white] (-2.5,-2.5) -- (-2.5,2.5) -- (2.5,2.5) -- (2.5,-2.5) -- (-2.5,-2.5);
\draw[thick, blue] (0,2) -- (0,-2);
\draw[white, fill=white] (0,0.38) circle (0.07cm); 
\draw[thick] plot [smooth cycle, tension=1] coordinates {(-1.5,1.2) (-0.5475, 0.625) (0.5,0.25) (-0.462, 0.805)};
\node[right] at (0.5, 0.25) {$v$};
\draw[white, fill=white] (0,0.6) circle (0.14cm);
\draw[white, fill=white] (-0.175,0.45) circle (0.07cm); 
\draw[thick, blue] (0,0.8) -- (0,0.5);
\draw[thick, green] (-1,2) arc (90:-90:1cm);
\draw[thick] plot [smooth cycle, tension=1] coordinates {(-1.5,1.2) (-0.952, 0.64) (-0.25,0) (-1.06, 0.55)};
\node[below] at (-0.25,0) {$u$};
\draw[white, fill=white] (-0.52,0.125) circle (0.052cm);
\draw[thick, green] (-1.5,0) .. controls (-1, 0) and (-0.75, -1) .. (-0.25,-1);
\draw[white, fill=white] (-0.84,-0.55) circle (0.08cm);
\draw[thick] plot [smooth cycle, tension=1] coordinates {(-1.5,1.2) (-1, 0.24) (-0.85,-0.85) (-1.1, 0.15)};
\draw[white, fill=white] (-1.06,0) circle (0.035cm);
\draw[white, fill=white] (-0.93,0) circle (0.035cm);
\draw[thick, green] (-1, 0) -- (-1.5, 0);
\draw[thick, green] (-1,0) arc (-90:-57:1cm);
\draw[white, fill=white] (-0.95,-0.4) circle (0.07cm);
\draw[white, fill=white] (-0.97,-0.32) circle (0.025cm);
\draw[white, fill=white] (-0.94,-0.47) circle (0.025cm);
\draw[thick, green] plot [smooth] coordinates {(-1.02,-0.315) (-0.95,-0.4) (-0.9,-0.47)};
\node[below] at (-0.85, -0.85) {$b$};
\draw[thick, green] (-0.25, -1) -- (2.5,-1);
\node [blue, above] at (0,2) {$A$};
\node [green] at (-1.25, 2.05) {$B$};
\draw [orange, fill=orange] (0,1) circle (0.05cm);
\draw [orange, fill=orange] (-1.5,0) circle (0.05cm);
\draw [orange, fill=orange] (0,-1) circle (0.05cm);
\node [orange] at (0, -1.5) {$\boldsymbol{\times}$};
\node [orange] at (0.5, -1.5) {$\mathrm{A}_1$};
\node [orange] at (-1.25, -1.75) {$\boldsymbol{\times}$};
\node [orange] at (-1.7, -1.75) {$\mathrm{A}_2$};
\draw[black, fill=black] (-1.5, 1.2) circle (0.05cm);
\end{tikzpicture}
\caption{A visualization of $\bbP(1,2,3)$}\label{fig:123}
\end{figure}
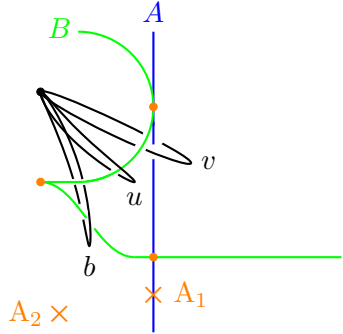

Consider the loops $b$, $u$, and $v$ around $A$ and $B$ pictured in Figure~\ref{fig:123}. One can show that
\[
\pi_1(\bbP(1,2,3) \ssm \{A,B\}) = \left\langle b,u,v\ |\ \mathrm{br}_2(b,v), \mathrm{br}_3(b,u), \mathrm{br}_4(u,v), (buv)^3 \right\rangle,
\]
where $\mathrm{br}_m(a,b)$ denotes the braid relation of length $m$:
\begin{equation}\label{eq:braid}
\mathrm{br}_m(a,b) = \begin{cases}
ab\cdots ab = ba \cdots ba & m~\textrm{even} \\
ab \cdots ba = ba \cdots ab & m~\textrm{odd}
\end{cases}
\end{equation}
It follows from standard facts about orbifolds and the procedure described in \cite{DeligneMostow1, MostowSigma} for computing orbifold weights that we obtain the presentation
\[
\Gam = \left\langle b,u,v\ |\ b^{r_1}, u^{r_1}, v^{r_2}, \mathrm{br}_2(b,v), \mathrm{br}_3(b,u), \mathrm{br}_4(u,v), (buv)^3 \right\rangle
\]
where $r_1$ and $r_2$ are computed directly from the set of weights $\mu$.

At the two crossings between $A$ and $B$, as well as at the singularity of $B$, the local orbifold group is then a complex reflection group
\[
p[m]q = \langle a,b~|~a^p, b^q, \mathrm{br}_m(ab) \rangle
\]
generated by complex reflections of order $p$ and $q$. When $p[m]q$ is finite (i.e., it is spherical), it admits a standard representation to $\U(2)$ where $\det(a) = \zeta_p$ and $\det(b) = \zeta_q$, each with eigenvalues $1$ and $\zeta_r$ for $r=p,q$. See \cite[\S 9.8]{Coxeter}. When $p[m]q$ is infinite it is Euclidean and the associated point of $\bbP(1,2,3)$ is a cusp point of the complex hyperbolic orbifold, that is, $\Gam \bs \bbB^2$ is the complement of the cusp points in $\bbP(1,2,3)$.

\medskip

\noindent
\textbf{Example 1: (5,4,1,1,1)/6} (nonuniform, arithmetic)

\medskip

This example is noncompact, arithmetic, and commensurable with the Picard modular group over the field $\bbQ(\zeta_3)$. The curve $A$ in Figure~\ref{fig:123} has orbifold weight $6$ and the curve $B$ has weight $3$. The point of tangency between $A$ and $B$ is a cusp with Euclidean cusp group $3[4]6$. See the appendix to \cite{StoverUrzua} for more on the orbifold structure and details concerning the computations that follow.

Our presentation is:
\[
\Gam = \left\langle b,u,v\ |\ b^3, u^3, v^6, \mathrm{br}_2(b,v), \mathrm{br}_3(b,u), \mathrm{br}_4(u,v), (buv)^3 \right\rangle
\]
With respect to the standard hermitian form $h$ in \S\ref{ssec:SUcoords}, we can choose generators for $\Gam < \PU(2,1)$ with the following representatives in $\U(2,1)$:
\begin{align*}
b_0 &\mapsto \begin{pmatrix}
1 & 0 & -1/\sqrt{-3} \\
0 & \zeta_6^5 & 0 \\
\sqrt{-3} & 0 & 0
\end{pmatrix} & u_0 \mapsto \begin{pmatrix}
\zeta_6^5 & 0 & 0 \\
\sqrt{-3} & \zeta_6 & 0 \\
\sqrt{-3} & \sqrt{-3} & \zeta_6^5
\end{pmatrix}\\
v_0 &\mapsto \begin{pmatrix}
\zeta_6 & 0 & 0 \\
0 & \zeta_3 & 0 \\
0 & 0 & \zeta_6
\end{pmatrix} &
\end{align*}
Scaling these by the cube root of their determinant, we obtain elements $\wh{b}, \wh{u}, \wh{v} \in \SU(2,1)$ that generate the preimage $\wh{\Gam}$ of $\Gam$ in $\SU(2,1)$.

Our generators in $\SU(2,1)$ satisfy the relations
\begin{align*}
\wh{b}^3 = \wh{u}^3 = \wh{v}^6 &= \wh{z}^2 \\
(\wh{b} \wh{u} \wh{v})^3 &= \Id
\end{align*}
where $\wh{z} = \zeta_3 \Id$ is the standard generator for the center of $\SU(2,1)$. Moreover, one checks that $\mathrm{br}_2(\wh{b}, \wh{v})$, $\mathrm{br}_3(\wh{b}, \wh{u})$, $\mathrm{br}_4(\wh{u}, \wh{v})$ all hold. The above relations give a presentation for $\wh{\Gam}$.

Now, using the methods developed in \S\ref{sec:Lift}, we can fix lifts $b$, $u$, and $v$ of our generators of $\Gam$ to the universal cover $\wt{G}$ of $\SU(2,1)$ and use computational software to compute the associated lift of each relation. For example, see Figures~\ref{fig:LiftbRelationPicard}-\ref{fig:LiftbRelationCPicard} for a graphical representation produced in Mathematica of the projection to $\bbC \ssm \{0\}$ as in \S\ref{sec:Lift} of the loop in $\SU(2,1)$ given by the relation $\wh{b}^9 = \Id$ with $\wh{b}^3$ a generator for the center of $\SU(2,1)$. Working through the lift of each relation, one obtains:
\begin{figure}[h]
\centering
\includegraphics[scale=0.15]{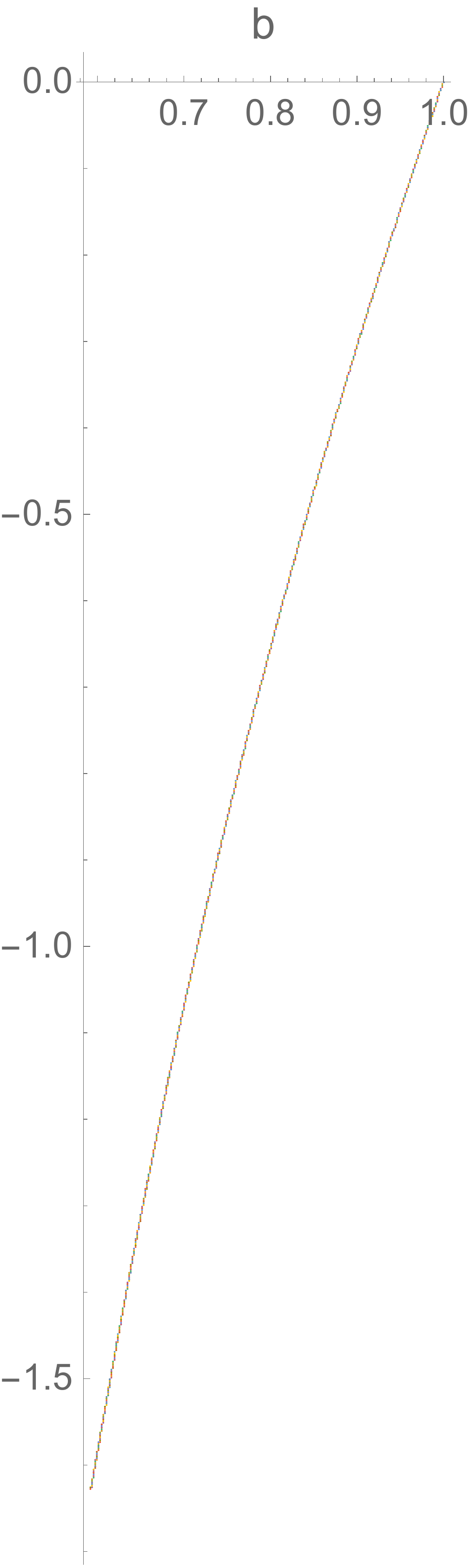}
\caption{Lifting $\wh{b}$ starts turning clockwise from $1 \in \bbC \ssm \{0\}$.}\label{fig:LiftbRelationPicard}
\end{figure}
\begin{figure}[h]
\centering
\includegraphics[scale=0.15]{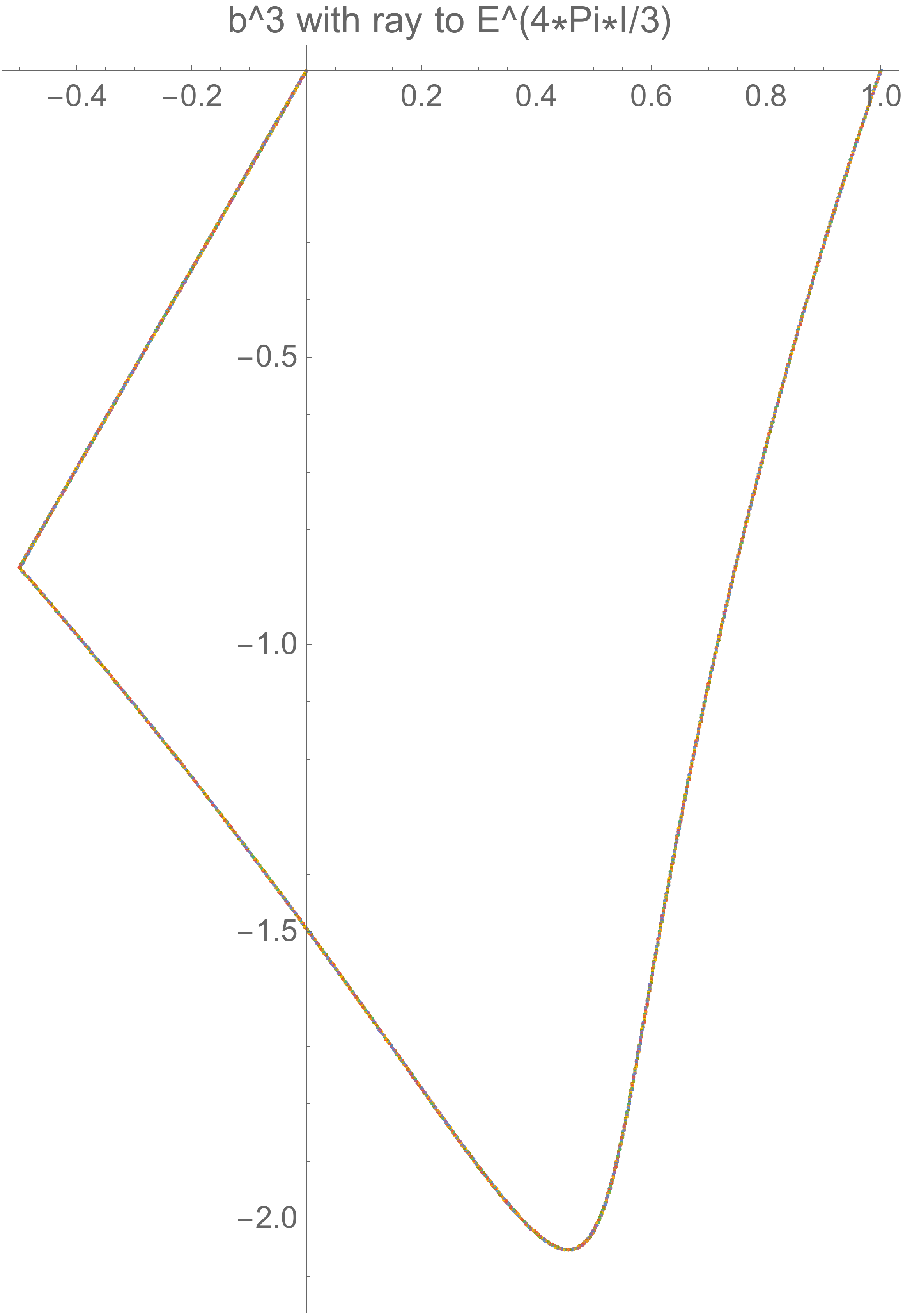}
\caption{The path for the lift of $\wh{b}^3$ ends at $e^{4 \pi i /3}$, so $\wh{b}^3$ lifts to $z^{-1}$.}\label{fig:LiftbRelationAPicard}
\end{figure}
\begin{figure}[h]
\centering
\includegraphics[scale=0.15]{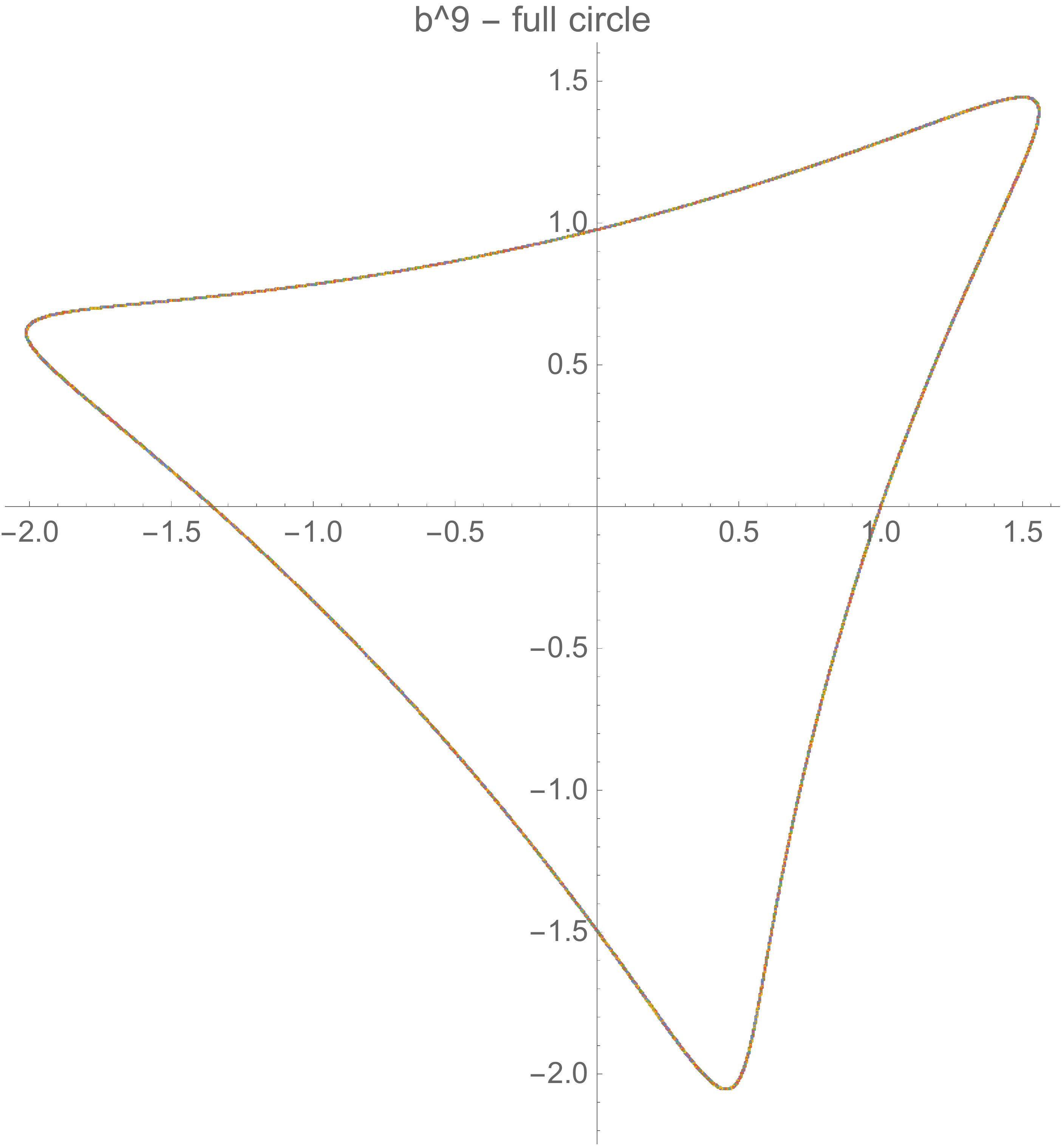}
\caption{A full clockwise rotation at $\wh{b}^9 = \Id$.}\label{fig:LiftbRelationCPicard}
\end{figure}

\begin{thm}\label{thm:PresentDMLiftPicard}
Let $\Gam < \PU(2,1)$ be the Deligne--Mostow lattice with weights $(5,4,1,1,1)/6$. The preimage $\wt{\Gam}$ of $\Gam$ in the universal cover $\wt{G}$ has presentation
\[
\left\langle b, u, v, z\ :\ b^3 z,\, u^3 z,\, v^6 z,\, \mathrm{br}_2(b,v),\, \mathrm{br}_3(b,u),\, \mathrm{br}_4(u,v),\, (buv)^3 z^3 \right\rangle,
\]
where $z$ is a lift to $\wt{G}$ of a generator for the center of $\SU(2,1)$ and hence $z^3$ denotes a generator for $\pi_1(\SU(2,1))$.
\end{thm}

We now explain how one proves that $\wt{\Gam}$ is residually finite. We exploit the subgroup of index $72$ associated with the ball quotient manifold constructed by Hirzebruch \cite{Hirzebruch}. Briefly, this is the fundamental group of one of the five ball quotient manifolds of Euler number $1$ that admit a smooth toroidal compactification \cite{StoverVolumes, DiCerboStoverClassify}. See \cite{Stoverb1, DiCerboStoverMultiple} for more on the geometry of this manifold and further applications to our understanding of the structure of ball quotients and their fundamental groups.

Let $\Lam < \Gam$ be the fundamental group of Hirzebruch's ball quotient. See \cite[Prop.\ A.8]{StoverUrzua} for details concerning how one can find $\Lam$, say in Magma, as a normal subgroup of $\Gam$ with index $72$. Let $\wt{\Lam}$ be the preimage of $\Lam$ in $\wt{\Gam}$, $\calN$ be the maximal two-step nilpotent quotient of $\Lam$, and $\wt{\calN}$ the maximal two-step nilpotent quotient of $\wt{\Lam}$ (recall, this is the quotient of $\wt{\Lam}$ by the second term of its lower central series). Using Magma, we see that there are exact sequences
\begin{align*}
1 \lra \bbZ^3 \lra &\calN \lra \bbZ^4 \lra 1 \\
1 \lra \bbZ^4 \lra &\wt{\calN} \lra \bbZ^4 \lra 1
\end{align*}
and that $z \in \wt{\Gam}$ maps to an infinite order element of the derived subgroup of $\wt{\calN}$. Corollary~\ref{cor:RFNilQuo} then implies that $\wt{\Lam}$ is residually finite, hence Lemma~\ref{lem:CommensurableRF} gives the following.

\begin{thm}\label{thm:PicardRF}
Let $\Gam < \PU(2,1)$ be any lattice commensurable with the Picard modular group over the field $\bbQ(\zeta_3)$ and $\wt{\Gam}$ be its preimage in the universal cover of $\PU(2,1)$. Then $\wt{\Gam}$ is residually finite and linear.
\end{thm}

Lattices that fall under the umbrella of Theorem~\ref{thm:PicardRF} include the Deligne--Mostow lattices with weights $(2,1,1,1,1)/3$, $(5,4,1,1,1)/6$, and seven others. It also covers the fundamental group of Hirzebruch's ball quotient and in fact every minimal volume complex hyperbolic $2$-manifold that admits a smooth toroidal compactification \cite{DiCerboStoverClassify}.

\medskip

\noindent
\textbf{Example 2: (11,7,2,2,2)/12} (uniform, arithmetic)

\medskip

This example is already covered by Theorem~\ref{thm:StoverRF}, and the details from this perspective are very similar to the previous example, so we are somewhat terse. We include it here because these methods give some more precise information about the relationships between the two-step nilpotent quotients of the fundamental group of the Stover surface and the associated central extension. For this example, the curve $A$ has orbifold weight $4$ and $B$ again has weight $3$.

For the hermitian form with matrix
\[
\begin{pmatrix}
1 + \sqrt{3} & -1 & 0 \\
-1 & -1 + \sqrt{3} & 0 \\
0 & 0 & -1
\end{pmatrix}
\]
we can choose the following representatives in $\U(2,1)$ for the generators:
\begin{align*}
b_0 &= \begin{pmatrix}
1 & 0 & 0 \\
(\frac{3}{2} + \sqrt{3}) - (1 + \frac{\sqrt{3}}{2}) i & \frac{1}{2}(-1 + i) (1 + \sqrt{3}) & -\frac{i}{2} (2 - i + \sqrt{3}) \\
\frac{1}{2}(1 + i) (1 + \sqrt{3}) & -1 - i & \frac{1}{2}(2 - i + \sqrt{3})
\end{pmatrix} \\
u_0 &= \begin{pmatrix}
\frac{1}{2}(1 + i) (1 + \sqrt{-3}) & 1 - \frac{\sqrt{3}+i}{2} & 0 \\
\frac{i}{2} (2 + i + \sqrt{3}) & \frac{1}{2}(\sqrt{3}-i) & 0 \\
0 & 0 & 1
\end{pmatrix} \\
v_0 &= \begin{pmatrix}
i & 0 & 0 \\
\frac{1}{2}(-1 + i) (1 + \sqrt{3}) & 1 & 0 \\
0 & 0 & 1
\end{pmatrix}
\end{align*}
To obtain elements of $\U(2,1)$ with respect to the standard hermitian form used in \S\ref{ssec:SUcoords}, we conjugate by the matrix
\[
\begin{pmatrix}
\sqrt{\frac{1}{2}(1+\sqrt{3})} & -\frac{1}{2}\sqrt{-1+\sqrt{3}} & -\frac{1}{\sqrt{2}} \\
0 & \frac{1}{\sqrt{1+\sqrt{3}}} & 0 \\
\sqrt{\frac{1}{2}(1+\sqrt{3})} & -\frac{1}{2}\sqrt{-1+\sqrt{3}} & \frac{1}{\sqrt{2}}
\end{pmatrix}
\]
and then scale by $\det^{1/3}$ to obtain elements $\wh{b}, \wh{u}, \wh{v} \in \SU(2,1)$ with respect to our standard hermitian form. Similarly to the previous example, we obtain:

\begin{thm}\label{thm:PresentDMLift}
Let $\Gam < \PU(2,1)$ be the Deligne--Mostow lattice with weights $(11,7,2,2,2)/12$. The preimage $\wt{\Gam}$ of $\Gam$ in the universal cover $\wt{G}$ has presentation
\[
\left\langle b, u, v, z\ :\ b^3 z,\, u^3 z,\, v^4 z,\, \mathrm{br}_2(b,v),\, \mathrm{br}_3(b,u),\, \mathrm{br}_4(u,v),\, (buv)^3 z^3 \right\rangle,
\]
where $z$ is a lift to $\wt{G}$ of a generator for the center of $\SU(2,1)$ and hence $z^3$ denotes a generator for $\pi_1(\SU(2,1))$.
\end{thm}

\begin{rem}
One can reinterpret Theorems~\ref{thm:PresentDMLiftPicard} and \ref{thm:PresentDMLift} as follows. With $q$ equal to $6$ or $4$, respectively, $\wt{\Gam}$ is generated by the natural preimages in $\wt{G}$ of the finite complex reflection subgroups
\begin{align*}
3[2]q &= \langle \conj{b}, \conj{v} \rangle &
3[3]3 &= \langle \conj{b}, \conj{u} \rangle &
3[4]q &= \langle \conj{u}, \conj{v} \rangle
\end{align*}
of $\PU(2,1)$, subject to the relation that $(buv)^3$ is a generator for $\pi_1(\SU(2,1))$. Here $\conj{b}, \conj{u}, \conj{v}$ denote generators in $\PU(2,1)$ and $b,u,v$ are lifts to the universal cover.
\end{rem}

Using Theorem~\ref{thm:PresentDMLift}, we now give a second proof of Theorem~\ref{thm:StoverRF}.

\begin{pf}[Second proof of Theorem~\ref{thm:StoverRF}]
Consider the fundamental group of the Stover surface, which is the congruence subgroup $\Lam \le \Gam$ of prime level dividing $3$ (see \cite{StoverHurwitz, DzambicRoulleau}). This is a normal subgroup of index $18144$ for which the preimage of $\Lam$ in $\SU(2,1)$ splits, i.e., is isomorphic to $\Lam \times \bbZ/3$.

Considering $\Lam$ as a torsion-free lattice in $\SU(2,1)$, we have a preimage $\wt{\Lam} < \wt{G}$ of $\Lam$ of the form
\[
1 \lra \langle z^3 \rangle \lra \wt{\Lam} \lra \Lam \lra 1
\]
with $\wt{\Lam}$ a subgroup of index $54432$ in $\wt{\Gam}$ and, as above, $z^3$ is a generator for $\pi_1(\SU(2,1))$. Using Magma, one can compute the maximal two-step nilpotent quotients $\calN$ and $\wt{\calN}$ of $\Lam$ and $\wt{\Lam}$, respectively. These fit into exact sequences:
\begin{align*}
1 \lra \bbZ / 4 \times \bbZ^{28} \lra &\calN \lra \bbZ^{14} \lra 1 \\
1 \lra \bbZ^{29} \lra &\wt{\calN} \lra \bbZ^{14} \lra 1
\end{align*}
Then $\langle z^3 \rangle \cong \bbZ$ maps injectively to $\wt{\calN}$, and is generated by the $4^{th}$ power of a certain primitive element of $\bbZ^{29}$. As in previous examples, Corollary~\ref{cor:RFNilQuo} implies that $\wt{\Lam}$ is residually finite and Lemma~\ref{lem:CommensurableRF} completes the proof of the theorem.
\end{pf}

\bibliography{CentralRF}

\end{document}